\documentclass[11pt,reqno]{amsart}
\usepackage{amsfonts}
\usepackage{amsmath,amsthm,amssymb,url,listings,amsfonts,amscd}
\usepackage[alphabetic,initials]{amsrefs}
\usepackage{mathrsfs}
\usepackage{lipsum}
\usepackage{bbding}
\usepackage{graphicx,latexsym}
\usepackage{hyperref}
\usepackage{scrtime}
\usepackage{color}

\newcommand{\bea}{\begin{eqnarray}}
\newcommand{\eea}{\end{eqnarray}}
\newcommand{\bna}{\begin{eqnarray*}}
\newcommand{\ena}{\end{eqnarray*}}

\numberwithin{equation}{section}

\theoremstyle{plain}
\newtheorem{lemma}{Lemma}[section]
\newtheorem{theorem}[lemma]{Theorem}

\theoremstyle{definition}

\newtheorem{remark}{Remark}

\begin{document}
	
\title{Uniform subconvex bounds for Rankin-Selberg $L$-functions}

\author{Qingfeng Sun}
\address{School of Mathematics and Statistics, Shandong University,
Weihai\\Weihai, Shandong 264209, China}
\email{qfsun@sdu.edu.cn}

\date{}

\begin{abstract}
Let $f$ be a Maass cusp form for $\rm SL_2(\mathbb{Z})$ with
Laplace eigenvalue $1/4+\mu_f^2$, $\mu_f>0$.
Let $g$ be an arbitrary but fixed holomorphic or Maass cusp form
for $\rm SL_2(\mathbb{Z})$. In this paper, we establish the
following uniform subconvexity
bound for the Rankin-Selberg $L$-function $L(s,f\otimes g)$
$$
L\left(1/2+it,f\otimes g\right)\ll (\mu_f+|t|)^{9/10+\varepsilon},
$$
where the implied constant depends only on $\varepsilon$ and $g$.

\end{abstract}
\thanks{(The author is partially
  supported by the National Natural Science Foundation
  of China (Grant Nos. 11871306 and 12031008)}
	
	
	\maketitle
	
\section{Introduction}\label{introduction}

The subconvexity problems of Rankin-Selberg $L$-functions in various aspects are
of great interest and importance and
have been intensively studied
by many authors (see \cite{KMV}, \cite{M}, \cite{HM}, \cite{LLY}, \cite{MV}, \cite{BR}, \cite{HT}, \cite{ASS}, \cite{BJN}, \cite{N} and the references therein). For the Rankin-Selberg $L$-function
$L(s,\pi_1\otimes \pi_2)$ associated to two irreducible cuspidal
automorphic representations $\pi_1$ and $\pi_2$ of $\rm GL_2$ with analytic conductor
$\mathcal{Q}(s,\pi_1\otimes \pi_2)$, the subconvexity problem of $L(s,\pi_1\otimes \pi_2)$
is aimed at obtaining estimates of the form
$L(s,\pi_1\otimes \pi_2)\ll \mathcal{Q}(s,\pi_1\otimes \pi_2)^{1/4-\delta}$
for some $\delta>0$ when $\mathrm{Re}(s)=1/2$, while the estimate
$L(s,\pi_1\otimes \pi_2)\ll \mathcal{Q}(s,\pi_1\otimes \pi_2)^{1/4+\varepsilon}$
with $\varepsilon>0$ arbitrarily small, which follows from the functional equation and
the Phragm\'{e}n-Lindel\"{o}f principle, is referred to as the convexity bound.

Let $f$ be a Maass cusp form for $\rm SL_2(\mathbb{Z})$ with normalized Fourier coefficients
$\lambda_f(n)$ and 
Laplace eigenvalue $1/4+\mu_f^2$, $\mu_f>0$.
Let $g$ be an arbitrary but fixed holomorphic or Maass cusp form
for $\rm SL_2(\mathbb{Z})$ with normalized Fourier coefficients
$\lambda_g(n)$.
We consider the Rankin-Selberg $L$-function
\bna
L\left(s,f\otimes g\right)=\zeta(2s)\sum_{n=1}^{\infty}
\frac{\lambda_f(n)\lambda_g(n)}{n^{s}},
\ena
where $\mbox{Re}(s)>1$ and $\zeta(s)$ is the Riemann zeta-function.
In this paper, we are concerned with
uniform subconvex estimates for the Rankin-Selberg $L$-function $L(s,f\otimes g)$
in the $t$ and $\mu_f$ aspects. In this case, the analytic conductor is
$(\mu_f+|t|+1)^2(|\mu_f-|t||+1)^2$ and the convexity bound is
$O\left((\mu_f+|t|+1)^{1/2+\varepsilon}(|\mu_f-|t||+1)^{1/2}\right)$.
The Lindel\"{o}f hypothesis asserts that
$$
L(1/2+it,f\otimes g) \ll_{g,\varepsilon} (\mu_f+|t|+1)^{\varepsilon}
$$
which is still out of reach at the present.
In 2006, Jutila and Motohashi \cite{JM} proved that
\bea\label{JM}
  L(1/2+it,f\otimes g) \ll_{g,\varepsilon} \left\{
  \begin{array}{ll}
    \mu_f^{2/3+\varepsilon}, & \textrm{for } 0\leq t\ll \mu_f^{2/3},\\
    \mu_f^{1/2}t^{1/4+\varepsilon}, & \textrm{for } \mu_f^{2/3}\leq t\ll t_f,\\
    t^{3/4+\varepsilon}, & \textrm{for } \mu_f \ll t\ll \mu_f^{3/2-\varepsilon},
  \end{array}\right.
\eea
which, however, does not cover all cases of $t$ and $\mu_f$.

Our main result states as follows.
\begin{theorem}\label{subconvexity} We have
\bna
L\left(1/2+it,f\otimes g\right)\ll (t+\mu_f)^{9/10+\varepsilon},
\ena
where the implied constant depends only on $\varepsilon$ and $g$.
\end{theorem}

Uniform subconvexity estimates in the $t$ and spectral aspect for $L$-values
have proven to be very difficult to establish
through current methods although there have been a few results.
For the $\rm GL_2$ $L$-function $L(s,f)$, Jutila and Motohashi \cite{JM1}
established by the moment method the uniform subconvexity bound
$$
  L(1/2+it,f)\ll_\varepsilon (\mu_f+|t|)^{1/3+\varepsilon},
$$
and then extended their result to $L\left(1/2+it,f\otimes g\right)$ (see \eqref{JM}).
Recently, Huang \cite{HB2} used the method of Munshi \cite{Mun1} to study the
$\rm GL_3\times \rm GL_2$ case and proved that
$$
  L(1/2+it,\pi\otimes f)\ll_{\pi,\varepsilon} (\mu_f+|t|)^{27/20+\varepsilon},
$$
where $f$ is as before and $\pi$ is a Hecke-Maass cusp form for $\rm SL_3(\mathbb{Z})$.
It is worth noting that in terms of $t$-aspect alone the best record bound
for $L\left(1/2+it,f\otimes g\right)$
is the Weyl type bound
$L\left(1/2+it,f\otimes g\right)\ll (1+|t|)^{2/3+\varepsilon}$ due to
Blomer, Jana and Nelson \cite{BJN} by combining in
a substantial way representation theory, local harmonic analysis, and analytic number theory.

\medskip
The paper is organized as follows. In Section \ref{sketch-of-proof}, we
provide a quick sketch and key steps of the proof. In
Section \ref{review-of-cuspform}, we review some basic materials of
automorphic forms on $ \rm GL_2$ and estimates on exponential integrals.
Sections \ref{details-of-proof}
and \ref{proofs-of-technical-lemma} give details of the proof for Theorem \ref{subconvexity}.

\bigskip
\noindent
{\bf Notation.}
Throughout the paper, the letters $\varepsilon$ and $A$ denote arbitrarily small and large
positive constants, respectively, not necessarily the same at each occurrence.
Implied constants may depend on $\varepsilon$ as well as on $g$.
The letters $q$, $m$ and $n$, with or without subscript,
denote integers.
We use $A\asymp B$ to mean that $c_1B\leq |A|\leq c_2B$ for some positive
constants $c_1$ and $c_2$ and the symbol
$q\sim C$ means $C<q\leq 2C$.

\section{Outline of the proof}\label{sketch-of-proof}

In this section, we provide a quick sketch of the proof for Theorem \ref{subconvexity}.
For simplicity, we assume $t>0$ and $t+\mu_f\asymp t-\mu_f\asymp T$.
By the approximate functional equation, we have
\bea\label{initial}
L\left(\frac{1}{2}+it,f\otimes g\right)\ll  T^{-1+\varepsilon}\mathcal{S},
\eea
where
\bna
\mathcal{S}=\sum_{n\sim T^2}\lambda_f(n)
\lambda_g(n)n^{-it}.
\ena
The first step is writing
\bna
\mathcal{S}=\sum_{n\sim T^2}\lambda_g(n)\sum_{m\sim T^2}\lambda_f(m)
m^{-it}\delta(m-n),
\ena
and using the $\delta$-method to detect the
Kronecker delta symbol $\delta(m-n)$. As in \cite{LS}, we use
the Duke-Friedlander-Iwaniec's $\delta$-method \eqref{DFI's} to write
\bea\label{before-voronoi}
\mathcal{S}&=&\frac{1}{Q}
\sum_{q\sim Q}\frac{1}{q}\int_{-T^{\varepsilon}}^{T^{\varepsilon}}
\sideset{}{^\star}\sum_{a\bmod{q}}\,
\sum_{n\sim T^2}\lambda_g(n)e\left(-\frac{na}{q}\right)
e\left(-\frac{n\zeta}{qQ}\right)\nonumber\\
&&\sum_{m\sim T^2}\lambda_f(m)e\left(\frac{ma}{q}\right)m^{-it}
e\left(\frac{m\zeta}{qQ}\right)\mathrm{d}\zeta,
\eea
where the $\star$ in the sum over $a$ means that the sum is restricted to $(a,q)=1$.

Next we use the Voronoi summation formulas to dualize the $m$ and $n$ sums.
The $n$-sum can be transformed into the following
\bea\label{GL2dual}
&&\sum_{n\sim T^2}\lambda_g(n)e\left(-\frac{na}{q}\right)
e\left(-\frac{n\zeta}{qQ}\right)\nonumber\\
&&\leftrightarrow T
\sum_{\pm}\sum_{n\sim T^2/Q^2}\lambda_g(n)e\left(\frac{n\bar{a}}{q}\right)
\Phi^{\pm}\left(n,q,\zeta\right),
\eea
where
 \bna
\Phi^{\pm}\left(n,q,\zeta\right)=\int_{x\asymp 1} x^{-1/4}
e\left(-\frac{\zeta T^2x}{qQ}\pm\frac{2\sqrt{mT^2x}}{q}\right)\mathrm{d}x.
\ena
Repeated integration by parts shows that
$\Phi^{-}\left(n,q,\zeta\right)\ll T^{-A}$ and a stationary phase analysis shows that
(note that $n\sim T^2/Q^2$)
\bea\label{plus case}
\Phi^+\left(n,q,\zeta\right)\asymp
\frac{q^{1/2}}{(nT^2)^{1/4}}
e\left(\frac{nQ}{q\zeta}\right)
U^{\natural}\left(\frac{nQ^2}{X\zeta^2}\right)
\asymp \frac{Q}{T}
e\left(\frac{nQ}{q\zeta}\right)
U^{\natural}\left(\frac{nQ^2}{T^2\zeta^2}\right).
\eea
Plugging \eqref{GL2dual} and \eqref{plus case} into \eqref{before-voronoi}, we are led to the sum
\bna
\mathcal{S}^*=
\sum_{q\sim Q}\frac{1}{q}\;
\sideset{}{^\star}\sum_{a\bmod{q}}\,
\sum_{n\sim T^2/Q^2}\lambda_g(n)e\left(\frac{n\bar{a}}{q}\right)\sum_{m\sim T^2}\lambda_f(m)e\left(\frac{ma}{q}\right)m^{-it}
\mathcal{K}(m,n,q),
\ena
where
\bna
\mathcal{K}(m,n,q)=\int_{-T^{\varepsilon}}^{T^{\varepsilon}}
U^{\natural}\left(\frac{nQ^2}{N\zeta^2}\right)
e\left(\frac{nQ}{q\zeta}+\frac{m\zeta}{qQ}\right)\mathrm{d}\zeta.
\ena
We evaluate the integral $\mathcal{K}(m,n,q)$ using the stationary phase method and get
\bna
\mathcal{K}(y;n,q)\asymp \frac{n^{1/4}q^{1/2}Q}{T^{3/2}}e\left(\frac{2\sqrt{mn}}{q}\right)
F\left(\frac{m}{T^2}\right)\asymp \frac{Q}{T}e\left(\frac{2\sqrt{mn}}{q}\right)
F\left(\frac{m}{T^2}\right)
\ena
for some smooth compactly supported function $F(y)$. Thus
\bea\label{before Voronoi-2}
\mathcal{S}^*&=&\frac{Q}{T}
\sum_{q\sim Q}\frac{1}{q}\;
\sideset{}{^\star}\sum_{a\bmod{q}}\,
\sum_{n\sim T^2/Q^2}\lambda_g(n)e\left(\frac{n\bar{a}}{q}\right)\nonumber\\
&&\times\sum_{m\sim T^2}\lambda_f(m)e\left(\frac{ma}{q}\right)m^{-it}
F\left(\frac{m}{T^2}\right)e\left(\frac{2\sqrt{mn}}{q}\right).
\eea

Applying the Voronoi formula to the sum over $m$, we have
\bea\label{GL2dual-2}
&&\sum_{m\sim T^2}\lambda_f(m)e\left(\frac{ma}{q}\right)m^{-it}
F\left(\frac{m}{T^2}\right)e\left(\frac{2\sqrt{mn}}{q}\right)\nonumber\\
&&\leftrightarrow
q\sum_\pm\sum_{m\asymp Q^2}\frac{\lambda_f(m)}{m}e\left(\pm\frac{\overline{a}m}{q}\right)
\Psi^{\pm}\left(\frac{m}{q^2},n,q\right),
\eea
where
\bna
\Psi^{\pm}\left(x,n,q\right)\asymp x^{it}(T^2x)^{1/2}
e\bigg(-\frac{T_1}{2\pi}\log\frac{T_1}{2e}
-\frac{T_2}{2\pi}\log\frac{|T_2|}{2e}+\frac{2\tau_0 n^{1/2}T}{q}\bigg).
\ena
with
$$
\tau_0:=\tau_0(m)=\left(T_1|T_2|/(4T^2m)\right)^{1/2}q\asymp 1.
$$
Then by plugging the dual sum \eqref{GL2dual-2} back into \eqref{before Voronoi-2} and
writing the Ramanujan sum
$S\left(m-n,0;q\right)$ as $\sum_{d|(m-n,q)}d\mu(q/d)$, we roughly get
\bea\label{before Cauchy}
\mathcal{S}^*&\approx&Qe\bigg(-\frac{T_1}{2\pi}\log\frac{T_1}{2e}
-\frac{T_2}{2\pi}\log\frac{|T_2|}{2e}\bigg)
\sum_{q\sim Q}\frac{1}{q^{1+2it}}\sum_{d|q}d\mu\left(\frac{q}{d}\right)
\nonumber\\
&&\times \sum_{m\asymp Q^2}\frac{\lambda_f(m)}{m^{1/2-it}}
\sum_{n\sim T^2/Q^2}\lambda_g(n)
e\left(2\tau_0 n^{1/2}T/q\right).
\eea

To prepare for an application of the Poisson summation in the $m$-variable,
we now apply the Cauchy-Schwarz inequality to smooth the $m$-sum
and put the $n$-sum inside the absolute value squared to get
\bna
\mathcal{S}^*&\ll&\frac{1}{Q}\sum_{q\sim Q}\,\sum_{d|q}d
\bigg(\sum_{m\sim Q^2}|\lambda_f(m)|^2\bigg)^{1/2}
\bigg(\sum_{m\sim Q^2}\bigg|\sum_{n\sim T^2/Q^2\atop n\equiv \pm m\bmod d}
\lambda_g(n)e\left(2\tau_0 n^{1/2}T/q\right)\bigg|^2\bigg)^{1/2}\\
&\ll&\sum_{q\sim Q}\,\sum_{d|q}d
\bigg(\sum_{m\sim Q^2}\bigg|\sum_{n\sim T^2/Q^2\atop n\equiv \pm m\bmod d}
\lambda_g(n)\mathfrak{I}^*\left(m,n,q\right)\bigg|^2\bigg)^{1/2}.
\ena
where
\bna
\mathfrak{I}^*\left(m,n,q\right)\asymp e\left(2\tau_0 n^{1/2}T/q\right).
\ena

\begin{remark}
If we open the absolute value squared, by the Rankin-Selberg estimate
for $\lambda_f(n)$ and the trivial estimate $\mathfrak{I}^*\left(m,n,q\right)\ll 1$,
the contribution from the diagonal term
$n=n'$ is given by
\begin{equation}\label{S-diagonal}
\begin{split}
\mathcal{S}_{\text{diag}}
\ll& \sum_{q\sim Q}\,\sum_{d|q}d
\bigg(\sum_{m\sim Q^2}\sum_{n\sim T^2/Q^2\atop n\equiv \pm m\bmod d}
|\lambda_g(n)|^2\left|\mathfrak{I}^*\left(m,n,q\right)\right|^2\bigg)^{1/2}\\
\ll& Q^{3/2}T,
\end{split}
\end{equation}
which will be fine for our purpose (i.e., $\mathcal{S}_{\text{diag}}=o(T^2)$)
as long as $Q\ll T^{2/3}$.
\end{remark}

Recall $\tau_0=\left(T_1|T_2|/(4Nm)\right)^{1/2}q\asymp 1$.
Note that the oscillation in the $n$-variable
 is of size $2\tau_0n^{1/2}T/q\approx T^2/Q^2$.
So opening the absolute value squared and applying the Poisson summation formula
in the $m$-variable, we have
\begin{equation*}
\begin{split}
\sum_{m\sim Q^2\atop m\equiv \pm n\bmod d}\mathfrak{I}^*\left(m,n,q\right)
\overline{\mathfrak{I}^*\left(m,n',q\right)}
\leftrightarrow \frac{Q^2}{d}
\sum_{\tilde{m}\ll \frac{dT^2/Q^2}{Q^2}}
\,\mathcal{H}\left(\frac{\tilde{m}Q^2}{d}\right),
\end{split}
\end{equation*}
where
\bea\label{correlation-integral}
\mathcal{H}(x)=\int_{\mathbb{R}}
\mathfrak{I}^*\left(Q^2\xi,n,q\right)
\overline{\mathfrak{I}^*\left(Q^2\xi,n',q\right)}
\, e\left(-x\xi\right)\mathrm{d}\xi.
\eea

The contribution to $\mathcal{S}^*$ from the zero-frequency $\tilde{m}=0$
will roughly correspond to the diagonal contribution $\mathcal{S}_{\text{diag}}$ in \eqref{S-diagonal}.
For the non-zero frequencies from the terms with $\tilde{m}\neq 0$, we note that
by performing stationary phase analysis, when $|x|$ is ``large", the expected estimate for
the triple integral $\mathcal{H}(x)$ in \eqref{correlation-integral} is
\bea\label{expect}
\mathcal{H}(x)\ll |x|^{-1/2},
\eea
which comes from the square-root cancellation of the two inner integrals
and the square-root cancellation in the $\xi$-variable. Note that this estimate does not hold for
``small" $|x|$. In fact, for these exceptional cases
the ``trivial" bound $\mathcal{H}(x)\ll 1$ will suffice for our purpose.
(These are the content of Lemma \ref{integral:lemma}). We ignore these exceptions and
plug the expected estimate \eqref{expect} for  $\mathcal{H}(x)$
into $\mathcal{S}$. It turns out
that the non-zero frequencies contribution $\mathcal{S}_{\text{off-diag}}$ from
$\tilde{m}\neq 0$ to $\mathcal{S}$ is given by
\bna
\mathcal{S}_{\text{off-diag}}&\ll&\sum_{q\sim Q}\,\sum_{d|q}d
\bigg(\sum_{n\sim T^2/Q^2}|\lambda_g(n)|^2
\sum_{n'\sim T^2/Q^2\atop n'\equiv n\bmod d}\frac{Q^2}{d}
\sum_{0\neq \widetilde{m}\ll dT^2/Q^4}\frac{d^{1/2}}{|\widetilde{m}|^{1/2}Q}\bigg)^{1/2}\\
&\ll&T^{5/2}/Q+T^{3/2}Q^{1/2}
\ll T^{5/2}/Q
\ena
provided that $Q<T^{2/3}$.
Hence combining this with the diagonal contribution $\mathcal{S}_{\text{diag}}$ in \eqref{S-diagonal}, we get
\bna
\mathcal{S}\ll  Q^{3/2}T+T^{5/2}/Q.
\ena
Plugging this estimate into \eqref{initial}, one has
\bna
L\left(\frac{1}{2}+it,f\otimes g\right)\ll  T^{\varepsilon}\left(Q^{3/2}+T^{3/2}/Q\right)
\ena
By choosing $Q=T^{3/5}$ we conclude that
\bna
L\left(\frac{1}{2}+it,f\otimes g\right)\ll T^{9/10+\varepsilon}.
\ena

\section{Preliminaries}\label{review-of-cuspform}

First we recall some basic results on automorphic forms for $\mathrm{GL}_2$.

\subsection{Holomorphic cusp forms for $\mathrm{GL}_2$}

Let $f$ be a holomorphic cusp form of weight $\kappa$ for $\rm SL_2(\mathbb{Z})$
with Fourier expansion
\bna
f(z)=\sum_{n=1}^{\infty}\lambda_f(n)n^{(\kappa-1)/2}e(nz)
\ena
for $\mbox{Im}\,z>0$, normalized such that $\lambda_f(1)=1$.
By the Ramanujan-Petersson conjecture proved by Deligne \cite{Del},
we have
$
\lambda_f(n)\ll \tau(n)\ll n^{\varepsilon}
$
with $\tau(n)$ being the divisor function.

For $h(x)\in C_c(0,\infty)$, we set
\bea\label{intgeral transform-1}
\Phi_h(x) =2\pi i^{\kappa} \int_0^{\infty} h(y) J_{\kappa-1}(4\pi\sqrt{xy})\mathrm{d}y,
\eea
where $J_{\kappa-1}$ is the usual $J$-Bessel function of order $\kappa-1$.
We have the following Voronoi summation formula (see \cite[Theorem A.4]{KMV}).

\begin{lemma}\label{voronoiGL2-holomorphic}
Let $q\in \mathbb{N}$ and $a\in \mathbb{Z}$ be such
that $(a,q)=1$. For $X>0$, we have
\bna\label{voronoi for holomorphic}
\sum_{n=1}^{\infty}\lambda_f(n)e\left(\frac{an}{q}\right)h\left(\frac{n}{N}\right)
=\frac{N}{q} \sum_{n=1}^{\infty}\lambda_f(n)
e\left(-\frac{\overline{a}n}{q}\right)\Phi_h\left(\frac{nN}{q^2}\right),
\ena
where $\overline{a}$ denotes
the multiplicative inverse of $a$ modulo $q$.
\end{lemma}

The function $\Phi_h(x)$ has the following asymptotic expansion
when $x\gg 1$ (see \cite{LS}, Lemma 3.2).

\begin{lemma}\label{voronoiGL2-holomorphic-asymptotic}
For any fixed integer $J\geq 1$ and $x\gg 1$, we have
\bna
\Phi_h(x)=x^{-1/4} \int_0^\infty h(y)y^{-1/4}
\sum_{j=0}^{J}
\frac{c_{j} e(2 \sqrt{xy})+d_{j} e(-2 \sqrt{xy})}
{(xy)^{j/2}}\mathrm{d}y
+O_{\kappa,J}\left(x^{-J/2-3/4}\right),
 \ena
where $c_{j}$ and $d_{j}$ are constants depending on $\kappa$.
\end{lemma}

\subsection{Maass cusp forms for $\mathrm{GL}_2$}

Let $f$ be a Hecke-Maass cusp form for $\rm SL_2(\mathbb{Z})$
with Laplace eigenvalue $1/4+\mu_f^2$. Then $f$ has a Fourier expansion
$$
f(z)=\sqrt{y}\sum_{n\neq 0}\lambda_f(n)K_{i\mu_f}(2\pi |n|y)e(nx),
$$
where $K_{i\mu}$ is the modified Bessel function of the third kind.
The Fourier coefficients satisfy
\bea\label{individual bound}
\lambda_f(n)\ll n^{\vartheta},
\eea
where, here and throughout the paper, $\theta$ denotes the exponent towards the Ramanujan
conjecture for $\rm GL_2$ Maass forms. The Ramanujan conjecture states that $\vartheta=0$ and
the current record due to Kim and Sarnak \cite{KS} is $\vartheta=7/64$.
On average we have the following
Rankin-Selberg estimate (see Proposition 19.6 in \cite{DFI2})
\bea\label{Rankin-Selberg}
\sum_{n\leq x}|\lambda_f(n)|^2\ll_{\varepsilon}x(x|\mu_f|)^{\varepsilon}.
\eea

For $h(x)\in C_c^{\infty}(0,\infty)$, we define the integral transforms
\bea\label{intgeral transform-2}
\begin{split}
\Phi_h^+(x) =& \frac{-\pi}{\sin(\pi i\mu_f)} \int_0^\infty h(y)\left(J_{2i\mu_f}(4\pi\sqrt{xy})
- J_{-2i\mu_f}(4\pi\sqrt{xy})\right) \mathrm{d}y,\\
\Phi_h^-(x) =& 4\varepsilon_f\cosh(\pi \mu_f)\int_0^\infty h(y)K_{2i\mu_f}(4\pi\sqrt{xy}) \mathrm{d}y,
\end{split}\eea
where $\varepsilon_f$ is an eigenvalue under the reflection operator.
We have the following Voronoi summation formula (see \cite[Theorem A.4]{KMV}).

\begin{lemma}\label{voronoiGL2-Maass}
Let $q\in \mathbb{N}$ and $a\in \mathbb{Z}$ be such
that $(a,q)=1$. For $X>0$, we have
\bna\label{voronoi for Maass form}
\sum_{n=1}^{\infty}\lambda_f(n)e\left(\frac{an}{q}\right)h\left(\frac{n}{N}\right)
= \frac{N}{q} \sum_{\pm}\sum_{n=1}^{\infty}\lambda_f(n)
e\left(\mp\frac{\overline{a}n}{q}\right)\Phi_h^{\pm}\left(\frac{nN}{q^2}\right),
\ena
where $\overline{a}$ denotes
the multiplicative inverse of $a$ modulo $q$.
\end{lemma}

For $x\gg 1$, we have (see (3.8) in \cite{LS})
\bea\label{The $-$ case}
\Phi_h^-(x)\ll_{\mu,A}x^{-A}.
\eea
For $\Phi_h^+(x)$ and $x\gg 1$, we have a similar asymptotic formula as for
$\Phi_h(x)$ in the holomorphic case (see \cite{LS}, Lemma 3.4).
\begin{lemma}\label{voronoiGL2-Maass-asymptotic}
For any fixed integer $J\geq 1$ and $x\gg 1$, we have
\bna
\Phi_h^{+}(x)=x^{-1/4} \int_0^\infty h(y)y^{-1/4}
\sum_{j=0}^{J}
\frac{c_{j} e(2 \sqrt{xy})+d_{j} e(-2 \sqrt{xy})}
{(xy)^{j/2}}\mathrm{d}y
+O_{\mu,J}\left(x^{-J/2-3/4}\right),
 \ena
where $c_{j}$ and $d_{j}$ are some constants depending on $\mu$.
\end{lemma}

\begin{remark}\label{decay-of-largeX}
For $x\gg X^{\varepsilon}$, we can choose $J$ sufficiently large so that
the contribution from the $O$-terms in Lemmas \ref{voronoiGL2-holomorphic-asymptotic} and
\ref{voronoiGL2-Maass-asymptotic}
is negligible. For the main terms
we only need to analyze the leading term $j=1$, as the analysis of the remaining
lower order terms is the same and their contribution is smaller
compared to that of the leading term.
\end{remark}

To deal with both $t$ and $\mu_f$ aspects, it is more convenient to use the
Voronoi formula of following form (see \cite{MS}, Eqs. (1.12),(1.15)).

\begin{lemma}\label{GL2 Voronoi formula-version 2}
Let $\varphi(x)\in C_c^{\infty}(0,\infty)$.
Let $a, \overline{a}, q\in\mathbb{Z}$ with $q\neq0, (a,q)=1$ and
$a\overline{a}\equiv1\;(\text{{\rm mod }} q)$. Then
\bna
\sum_{n\geq 1}\lambda_f(n)e\left(\frac{an}{q}\right)\varphi(n)
=q\sum_\pm\sum_{n\geq 1}\frac{\lambda_f(n)}{n}e\left(\pm\frac{\overline{a}n}{q}\right)
\Psi_{\varphi}^{\pm}\left(\frac{n}{q^2}\right),
\ena
where for $\sigma>-1$,
\bea\label{integral 1}
\Psi_{\varphi}^{\pm}(x)=\frac{1}{4\pi^2 i}\int_{(\sigma)}(\pi^2 x)^{-s}\gamma_f^{\pm}(s)\widetilde{\varphi}(-s) \mathrm{d}s,
\eea
with
\bea\label{gamma}
\gamma_f^{\pm}(s)=\prod_\pm\frac{\Gamma\left(\frac{1+s\pm i\mu_f}{2}\right)}
{\Gamma\left(\frac{-s\pm i\mu_f}{2}\right)}\pm
\prod_\pm\frac{\Gamma\left(\frac{2+s\pm i\mu_f}{2}\right)}
{\Gamma\left(\frac{1-s\pm i\mu_f}{2}\right)}.
\eea
Here $\widetilde{\varphi}(s)=\int_0^{\infty}\varphi(u)u^{s-1}\mathrm{d}u$ is the Mellin transform of $\varphi$.
\end{lemma}
By Stirling asymptotic formula (see \cite{O}, Section 8.4, in particular (4.03)),
for $|\arg s|\leq \pi-\varepsilon$ for any $\varepsilon>0$ and $|s|\gg 1$,
\bna
\ln \Gamma(s)=\left(s-\frac{1}{2}\right)\ln s-s+\frac{1}{2}\ln(2\pi)
+\sum_{j=1}^{K_1}\frac{B_{2j}}{2j(2j-1)s^{2j-1}}+O_{K_1,\varepsilon}\left(\frac{1}{|s|^{2K_1+1}}\right),
\ena
where $B_j$ are Bernoulli numbers. Thus for $s=\sigma+it$, $\sigma$ fixed and $|\tau|\geq 2$,
\bea\label{Stirling approximation}
\Gamma(\sigma+i\tau)=\sqrt{2\pi}(i\tau)^{\sigma-1/2}
e^{-\pi|\tau|/2}\left(\frac{|\tau|}{e}\right)^{i\tau}
\left(1+\sum_{j=1}^{K_2}\frac{c_j}{\tau^j}+O_{\sigma,K_2,\varepsilon}
\bigg(\frac{1}{|\tau|^{K_2+1}}\bigg)\right),
\eea
where the constants $c_j$ depend on $j,\sigma$ and $\varepsilon$.
Thus for $\sigma\geq -1/2$,
\bea\label{Gamma-2}
\gamma_f^{\pm}(\sigma+i\tau)&=&\prod_\pm\frac{\Gamma(\frac{1+\sigma+i(\tau\pm \mu_f)}{2})}
{\Gamma(\frac{-\sigma-i(\tau\pm \mu_f)}{2})}\pm
\prod_\pm\frac{\Gamma(\frac{2+\sigma+i(\tau\pm \mu_f)}{2})}
{\Gamma(\frac{1-\sigma-i(\tau\pm \mu_f)}{2})}\nonumber\\
&\ll &(|\tau+\mu_f||\tau-\mu_f|)^{\sigma+1/2}.
\eea

\subsection{Rankin-Selberg $L$-function}
Let $f$ be a Maass cusp form for $\rm SL_2(\mathbb{Z})$ with
Laplace eigenvalue $1/4+\mu_f^2$, $\mu_f>0$, and parity $\delta_f=0$ or 1.
Let $g$ be either a holomorphic cusp form of weight $2\kappa$ or a Maass cusp form  for $\rm SL_2(\mathbb{Z})$
with Laplace eigenvalue $1/4+\mu_g^2$, $\mu_g>0$, and parity $\delta_g=0$ or 1.
For $\mbox{Re}(s)>1$, the Rankin-Selberg $L$-function is defined as
\bna
L\left(s,f\otimes g\right)=\zeta(2s)\sum_{n=1}^{\infty}
\frac{\lambda_f(n)\lambda_g(n)}{n^{s}},
\ena
which can be meromorphically continued to the whole complex plane except for
a simple pole at $s=1$ if $g=\overline{f}$ and satisfies the functional equation (see \cite{IK}, Section 5.11 and \cite{JM}, Lemma 1)
\bea\label{functional equation}
\Lambda(s,f\otimes g)=\Lambda(1-s,f\otimes g),
\eea
where
\bna
\Lambda(s,f\otimes g)=\gamma(s,f\otimes g)L\left(s,f\otimes g\right)
\ena
with
\bna\label{Gamma factor-1}
\gamma(s,f\otimes g)&=&(2\pi)^{-2s} \Gamma\Big(s+\frac{\kappa-1}{2}+i\mu_f\Big)
\Gamma\Big(s+\frac{\kappa-1}{2}-i\mu_f\Big)\quad \mathrm{for}\,g\,\mathrm{holomorphic},\\
\gamma(s,f\otimes g)&=&\pi^{-2s} \Gamma\Big(\frac{s+\delta+i(\mu_f+\mu_g)}{2}\Big)
\Gamma\Big(\frac{s+\delta+i(\mu_f-\mu_g)}{2}\Big)\nonumber\\
&&\times\Gamma\Big(\frac{s+\delta-i(\mu_f+\mu_g)}{2}\Big)
\Gamma\Big(\frac{s+\delta-i(\mu_f-\mu_g)}{2}\Big)\quad \mathrm{for}\,g\,\mathrm{a \,Maass \,form}.
\ena
Here $\delta=0$ or $1$ according to whether $\delta_f=\delta_g$ or not.

\subsection{Estimates for exponential integrals}

Let
\begin{equation*}
 I = \int_{\mathbb{R}} w(y) e^{i \varrho(y)} dy.
\end{equation*}
Firstly, we have the following estimates for exponential integrals
(see \cite[Lemma 8.1]{BKY}  and \cite[Lemma A.1]{AHLQ}).
	
	\begin{lemma}\label{lem: upper bound}
		Let $w(x)$ be a smooth function    supported on $[ a, b]$ and
        $\varrho(x)$ be a real smooth function on  $[a, b]$. Suppose that there
		are   parameters $Q, U,   Y, Z,  R > 0$ such that
		\begin{align*}
		\varrho^{(i)} (x) \ll_i Y / Q^{i}, \qquad w^{(j)} (x) \ll_{j } Z / U^{j},
		\end{align*}
		for  $i \geqslant 2$ and $j \geqslant 0$, and
		\begin{align*}
		| \varrho' (x) | \geqslant R.
		\end{align*}
		Then for any $A \geqslant 0$ we have
		\begin{align*}
		I \ll_{ A} (b - a)
Z \bigg( \frac {Y} {R^2Q^2} + \frac 1 {RQ} + \frac 1 {RU} \bigg)^A .
		\end{align*}
			\end{lemma}

Next, we need the following evaluation for exponential integrals
which are
 Lemma 8.1 and Proposition 8.2 of \cite{BKY} in the language of inert functions
 (see \cite[Lemma 3.1]{KPY}).

Let $\mathcal{F}$ be an index set, $Y: \mathcal{F}\rightarrow\mathbb{R}_{\geq 1}$ and under this map
$T\mapsto Y_T$
be a function of $T \in \mathcal{F}$.
A family $\{w_T\}_{T\in \mathcal{F}}$ of smooth
functions supported on a product of dyadic intervals in $\mathbb{R}_{>0}^d$
is called $Y$-inert if for each $j=(j_1,\ldots,j_d) \in \mathbb{Z}_{\geq 0}^d$
we have
\bna
C(j_1,\ldots,j_d)
= \sup_{T \in \mathcal{F} } \sup_{(y_1, \ldots, y_d) \in \mathbb{R}_{>0}^d}
Y_T^{-j_1- \cdots -j_d}\left| y_1^{j_1} \cdots y_d^{j_d}
w_T^{(j_1,\ldots,j_d)}(y_1,\ldots,y_d) \right| < \infty.
\ena

\begin{lemma}
\label{lemma:exponentialintegral}
 Suppose that $w = w_T(y)$ is a family of $Y$-inert functions,
 with compact support on $[Z, 2Z]$, so that
$w^{(j)}(y) \ll (Z/X)^{-j}$.  Also suppose that $\varrho$ is
smooth and satisfies $\varrho^{(j)}(y) \ll Y/Z^j$ for some
$H/X^2 \geq R \geq 1$ and all $y$ in the support of $w$.
\begin{enumerate}
 \item
 If $|\varrho'(y)| \gg Y/Z$ for all $y$ in the support of $w$, then
 $I \ll_A Z R^{-A}$ for $A$ arbitrarily large.
 \item If $\varrho''(y) \gg Y/Z^2$ for all $y$ in the support of $w$,
 and there exists $y_0 \in \mathbb{R}$ such that $\varrho'(y_0) = 0$ (note $y_0$ is
 necessarily unique), then
 \begin{equation}
  I = \frac{e^{i \varrho(y_0)}}{\sqrt{\varrho''(y_0)}}
 F(y_0) + O_{A}(  Z R^{-A}),
 \end{equation}
where $F(y_0)$ is an $Y$-inert function (depending on $A$)  supported
on $y_0 \asymp Z$.
\end{enumerate}
\end{lemma}

We also need the
second derivative test (see \cite[Lemma 5.1.3]{Hux2}).
	
\begin{lemma}\label{lem: 2st derivative test, dim 1}
Let $\varrho(x)$ be real and twice
differentiable on the open interval $[a, b]$
with $ \varrho'' (x) \gg \lambda_0>0$  on $[a, b]$. Let $w(x)$
be real on $[ a, b]$ and let $V_0$ be its total
variation on $[ a, b]$ plus the maximum modulus of $w(x)$ on $[ a, b]$.
Then
		\begin{align*}
	I\ll \frac {V_0} {\sqrt{\lambda_0}}.
		\end{align*}
	\end{lemma}

\section{Proof of the main theorem}\label{details-of-proof}

\subsection{Reduction}
Without loss of generality, we assume $t>0$.
In view of \eqref{functional equation}, by the approximate functional equation (see \cite{IK}, Theorem 5.3, Proposition 5.4), we have
\bna
L\left(\frac{1}{2}+it,f\otimes g\right)\ll T_1^{\varepsilon}
\sup_{ N\ll T_1^{1+\varepsilon}|T_2|}\frac{1}{\sqrt{N}}
\bigg|\sum_{n\geq 1}\lambda_f(n)
\lambda_g(n)n^{-it}V\left(\frac{n}{N}\right)\bigg|+ 1,
\ena
where $V\in C_c^{\infty}(1,2)$ satisfying $V^{(j)}(x)\ll_j 1$ for any integer $j\geq 0$, and
\bea\label{T1-T2}
T_1=t+\mu_f, \qquad T_2=t-\mu_f.
\eea
By Cauchy-Schwarz inequality and \eqref{Rankin-Selberg}, we have
\bna
&&\sup_{ N\ll T_1^{9/5}}\frac{1}{\sqrt{N}}
\bigg|\sum_{n\geq 1}\lambda_f(n)
\lambda_g(n)n^{-it}V\left(\frac{n}{N}\right)\bigg|\\
&\ll&\sup_{N\ll T_1^{9/5}}\frac{1}{\sqrt{N}}\Big(\sum_{n\leq N}|\lambda_f(n)|^2
\Big)^{1/2}\Big(\sum_{n\leq N}|\lambda_g(n)|^2
\Big)^{1/2}\\
&\ll&\sup_{1\ll N\ll T_1^{9/5}}N^{1/2}\ll T_1^{9/10}.
\ena
It follows that
\bea\label{step 1}
L\left(\frac{1}{2}+it,f\otimes g\right)\ll  T_1^{\varepsilon}
\sup_{ T_1^{9/5}\ll N\ll T_1^{1+\varepsilon}|T_2|}\frac{1}{\sqrt{N}}
|\mathcal{S}(N)|+ T_1^{9/10},
\eea
where
\bea\label{main sum}
\mathcal{S}(N)=\sum_{n\geq 1}\lambda_f(n)
\lambda_g(n)n^{-it}V\left(\frac{n}{N}\right).
\eea
Note that the first term above is vanishing unless
\bea\label{T2 condition}
|T_2|\gg T_1^{4/5-\varepsilon},
\eea
which we shall henceforth assume.
Note that the trivial upper bound of $\mathcal{S}(N)$ is $O(N)$
by Cauchy-Schwarz inequality and \eqref{Rankin-Selberg}.
In the rest of the paper,
we are devoted to proving a nontrivial estimate for $\mathcal{S}(N)$.

\subsection{Duke-Friedlander-Iwaniec $\delta$-method}

Let \bna
\delta(n)=
\left\{
  \begin{array}{ll}
    1 \quad & \mathrm{if} \; n=0,\\
    0 \quad & \mathrm{otherwise}.
  \end{array}
\right.
\ena
The $\delta$-method of
Duke, Friedlander and Iwaniec (see \cite[Chapter 20]{IK})
states that for any $n\in \mathbb{Z}$ and $Q\in \mathbb{R}^+$,
\bea\label{DFI's}
\delta(n)=\frac{1}{Q}\sum_{1\leq q\leq Q} \;\frac{1}{q}\;
\sideset{}{^\star}\sum_{a\bmod{q}}e\left(\frac{na}{q}\right)
\int_\mathbb{R}g(q,\zeta) e\left(\frac{n\zeta}{qQ}\right)\mathrm{d}\zeta,
\eea
where the $\star$ on the sum indicates
that the sum over $a$ is restricted to $(a,q)=1$.
The function $g$ has the following properties (see (20.158) and (20.159)
of \cite{IK} and  Lemma 15 of \cite{HB})
\bea\label{g-h}
g(q,\zeta)\ll |\zeta|^{-A}, \;\;\;\;\;\; g(q,\zeta) =1+
O\left(\frac{Q}{q}\left(\frac{q}{Q}+|\zeta|\right)^A\right)
\eea
for any $A>1$ and
\bea\label{g rapid decay}
\frac{\partial^j}{\partial \zeta^j}g(q,\zeta)\ll
|\zeta|^{-j}\min\bigg(|\zeta|^{-1},\frac{Q}{q}\bigg)\log Q, \quad j\geq 1.
\eea

We write \eqref{main sum} as
\bna
\mathcal{S}(N)=\sum_{n\geq 1}\lambda_g(n)
U\left(\frac{n}{N}\right)\sum_{m\geq 1}
\lambda_f(m)m^{-it}
V\left(\frac{m}{N}\right)\delta(m-n),
\ena
where $U(x)\in \mathcal{C}_c^{\infty}(1/2,5/2)$ satisfying $U(x)=1$
for $x\in [1,2]$ and $U^{(j)}(x)\ll_j 1$ for any integer $j\geq 0$.
Plugging the identity \eqref{DFI's} for
$\delta(m-n)$ in and
exchanging the order of integration and summations,
we get
\bna
\mathcal{S}(N)&=&\frac{1}{Q}\sum_{1\leq q\leq Q}\frac{1}{q}
\int_{\mathbb{R}} g(q,\zeta)\;
\sideset{}{^\star}\sum_{a\bmod{q}}\;
\sum_{n\geq 1}\lambda_{g}(n)
e\left(-\frac{na}{q}\right)U\left(\frac{n}{N}\right)
e\left(-\frac{n\zeta}{qQ}\right)\nonumber\\
&& \sum_{m\geq 1}\lambda_{f}(m)e\left(\frac{ma}{q}\right)
m^{-it}V\left(\frac{m}{N}\right)
e\left(\frac{m\zeta}{qQ}\right)\mathrm{d}\zeta.
\ena
Note that the contribution from $|\zeta|\leq N^{-G}$ is negligible for $G>0$ sufficiently large.
Moreover, by the first property in \eqref{g-h}, we can restrict $\zeta$ in the range
$|\zeta|\leq N^{\varepsilon}$ up to an negligible error. So we can
insert a smooth partition of unity for the $\zeta$-integral and write $\mathcal{S}(N)$ as
\bna
&&\sum_{N^{-G}\ll \Xi\ll N^{\varepsilon}\atop \text{dyadic}}\frac{1}{Q}
 \sum_{1\leq q\leq Q}\frac{1}{q}
\int_{\mathbb{R}}g(q,\zeta)W\left(\frac{|\zeta|}{\Xi}\right)\;
\sideset{}{^\star}\sum_{a\bmod{q}}\;
\sum_{n\geq 1}\lambda_g(n)
e\left(-\frac{na}{q}\right)U\left(\frac{n}{N}\right)
e\left(-\frac{n\zeta}{qQ}\right)\nonumber\\
&& \times\sum_{m\geq 1}\lambda_f(m)e\left(\frac{ma}{q}\right)m^{-it}
V\left(\frac{m}{N}\right)
e\left(\frac{m\zeta}{qQ}\right)\mathrm{d}\zeta
+O_A(N^{-A}),
\ena
where $W(x)\in \mathcal{C}_c^{\infty}(1,2)$
 satisfying $W^{(j)}(x)\ll_j 1$ for any integer $j\geq 0$.
 Without loss of generality,
we only consider the contribution from $\zeta>0$ (the proof for $\zeta<0$ is entirely
similar). By abuse of notation, we still write the contribution from $\zeta>0$
as $\mathcal{S}(N)$.

Next we break the $q$-sum $\sum_{1\leq q\leq Q}$ into dyadic segments $q\sim C$ with $1\ll C\ll Q$ and write
\bea\label{C range}
\mathcal{S}(N)=\sum_{N^{-G}\ll \Xi\ll X^{\varepsilon}\atop \text{dyadic}}
\sum_{1\ll C\ll Q\atop \text{dyadic}}\mathscr{S}(C,\Xi)+O_A(N^{-A}),
\eea
where $\mathscr{S}(C,\Xi)=\mathscr{S}(N,C,\Xi)$ is
\bea\label{beforeVoronoi}
\mathscr{S}(C,\Xi)&=&\frac{1}{Q}\sum_{q\sim C}\frac{1}{q}
\int_{\mathbb{R}}g(q,\zeta)W\left(\frac{\zeta}{\Xi}\right)\;
\sideset{}{^\star}\sum_{a\bmod{q}}
\sum_{n\geq 1}\lambda_g(n)
e\left(-\frac{na}{q}\right)U\left(\frac{n}{N}\right)
e\left(-\frac{n\zeta}{qQ}\right)\nonumber\\
&& \sum_{m\geq 1}\lambda_f(m)e\left(\frac{ma}{q}\right)m^{-it}
V\left(\frac{m}{N}\right)
e\left(\frac{m\zeta}{qQ}\right)\mathrm{d}\zeta.
\eea

\subsection{Applying Voronoi summation formulas}
In this subsection, we shall apply Voronoi summation formulas
to the $n$-and $m$-sums in \eqref{beforeVoronoi}.
We first consider the sum over $n$.
Depending on whether $g$ is holomorphic or Maass, we apply
Lemma \ref{voronoiGL2-holomorphic}
or Lemma \ref{voronoiGL2-Maass} respectively with $h(x)=U(x)e\left(-\zeta Nx/(qQ)\right)$,
to transform the $n$-sum in \eqref{beforeVoronoi} into
\bea\label{$n$-sum after GL2 Voronoi}
\frac{N}{q}\sum_{\pm}\sum_{n\geq 1}\lambda_g(n)e\left(\pm \frac{n\overline{a}}{q}\right)
\Phi_{h}^{\pm}\left(\frac{nN}{q^2}\right),
\eea
where if $g$ is holomorphic, $\Phi_{h}^+(x)=\Phi_{h}(x)$ with $\Phi_{h}(x)$
given by \eqref{intgeral transform-1} and $\Phi_{h}^-(x)=0$,
while for $g$ a Hecke--Maass cusp form,
$\Phi_{h}^{\pm}(x)$ are given by \eqref{intgeral transform-2}.

Assume
\bea\label{assumption 1}
Q<N^{1/2-\varepsilon}.
\eea
Then we have $nN/q^2\gg N^{\varepsilon}$. In particular,
by \eqref{The $-$ case}, the contribution from
$\Phi_{h}^{-}\left(nN/q^2\right)$ is $O_A(N^{-A})$.
For $\Phi_{h}^{+}\left(nN/q^2\right)$ we apply Lemma \ref{voronoiGL2-holomorphic-asymptotic},
Lemma \ref{voronoiGL2-Maass-asymptotic}
and Remark \ref{decay-of-largeX} and find that evaluating of
the sum \eqref{$n$-sum after GL2 Voronoi} is reduced to dealing with the sum
 \bea\label{integral 2}
 \frac{N^{3/4}}{q^{1/2}}\sum_{\pm}\sum_{n\geq 1}\frac{\lambda_g(n)}{n^{1/4}}
e\left(\frac{n\overline{a}}{q}\right)\Phi^{\pm}\left(n,q,\zeta\right),
 \eea
 where
 \bea\label{Psi definition}
\Phi^{\pm}\left(n,q,\zeta\right)=\int_0^\infty U(x)x^{-1/4}
e\left(-\frac{\zeta Nx}{qQ}\pm \frac{2\sqrt{nNx}}{q}\right)\mathrm{d}x.
\eea
 Note that by \eqref{assumption 1}, the first derivative of the phase function
 in $\Phi^{-}$ is
 \bna
-\frac{\zeta N}{qQ}-\frac{\sqrt{nN/x}}{q}\gg N^{\varepsilon}.
 \ena
 By applying integration by parts
 repeatedly, one finds that the contribution from
$\Phi^{-}$ is negligible. Moreover, for $\zeta\asymp \Xi, N\Xi/(CQ)\ll N^{\varepsilon}$,
the first derivative of the phase function
 in $\Phi^{+}$ is
$$
-\frac{\zeta N}{qQ}+\frac{\sqrt{nN/x}}{q}\gg N^{\varepsilon}
$$
which implies the contributions from these $\zeta$ for $\Phi^{+}$ are
also negligible. So in the following, we only need to consider
$\Phi^{+}\left(n,q,\zeta\right)$ with $\zeta$ in the range
\bea\label{zeta range}
\zeta\asymp \Xi, \quad N\Xi/(CQ)\gg N^{\varepsilon}.
\eea
In this case, we
apply a stationary phase analysis to $\Phi^{+}$.
The stationary point $x_0$ is given by $x_0=nQ^2/(N\zeta^2)$.
Applying Lemma \ref{lemma:exponentialintegral} (2) with $X=Z=1$ and
$Y=R=\sqrt{nN}/q\gg N^{\varepsilon}$, we obtain
\bna
\Phi^+\left(n,q,\zeta\right)
=\frac{q^{1/2}}{(nN)^{1/4}}
e\left(\frac{nQ}{q\zeta}\right)
U^{\natural}\left(\frac{nQ^2}{N\zeta^2}\right)+O_A\left(N^{-A}\right),
\ena
where $U^{\natural}$ is an $1$-inert function (depending on $A$)
supported on $x_0 \asymp 1$.
In particular, this implies, up to a negligible error, we only need to consider those $n$
in the range $n\asymp  N\Xi^2/Q^2$.
Plugging the above asymptotic formula for $\Phi^+\left(n,q,\zeta\right)$ and \eqref{integral 2}
into \eqref{beforeVoronoi}
and switching the order of integrations and summations, we are led to the sum
 \bea\label{after first Voronoi}
\mathscr{S}^*(C,\Xi):=&&\frac{N^{1/2}}{Q}\sum_{q\sim C}\frac{1}{q}\;
\sideset{}{^\star}\sum_{a\bmod{q}}\,
\sum_{n\asymp  \frac{N\Xi^2}{Q^2}}\frac{\lambda_g(n)}{n^{1/2}}
e\left(\frac{n\overline{a}}{q}\right)
\nonumber\\
&& \times\sum_{m\geq 1}\lambda_f(m)e\left(\frac{ma}{q}\right)m^{-it}
V\left(\frac{m}{N}\right)\mathcal{K}(m,n,q,\Xi),
\eea
where $\mathcal{K}(m,n,q,\Xi)$ is given by
\bna
\mathcal{K}(m,n,q,\Xi)=\int_{\mathbb{R}}g(q,\zeta)W\left(\frac{\zeta}{\Xi}\right)
U^{\natural}\left(\frac{nQ^2}{N\zeta^2}\right)
e\left(\frac{nQ}{q\zeta}+\frac{m\zeta}{qQ}\right)\mathrm{d}\zeta.
\ena

Next, we derive an asymptotic expansion for $G(m,n,q,\Xi)$.
By making a change of variable $nQ^2/(N\zeta^2)\rightarrow \zeta$,
\bna
\mathcal{K}(m,n,q,\Xi)
=\frac{n^{1/2}Q}{N^{1/2}}
\int_0^{\infty}\phi(\zeta)
\exp\big(i\varpi(\zeta)\big)\mathrm{d}\zeta,
\ena
where
\bna
\phi(\zeta):=-\frac{1}{2}\zeta^{-3/2}U^{\natural}(\zeta)
g\bigg(q,\frac{n^{1/2}Q}{\zeta^{1/2}N^{1/2}}\bigg)
W\bigg(\frac{n^{1/2}Q}{\zeta^{1/2}N^{1/2}\Xi}\bigg)
\ena
and the phase function $\varpi(\zeta)$ is given by
\bna
\varpi(\zeta)=\frac{2\pi n^{1/2}N^{1/2}}{q}
\left(\frac{m}{N}\zeta^{-1/2}+\zeta^{1/2}\right).
\ena
Note that
\bna
\varpi'(\zeta)=
\frac{\pi n^{1/2}N^{1/2}}{q}
\left(-\frac{m}{N}\zeta^{-3/2}+\zeta^{-1/2}\right),
\ena
and for $j\geq 2$,
\bna
\varpi^{(j)}(\zeta)
=\left(-\frac{3}{2}\right)\cdot\cdot\cdot\left(\frac{1}{2}-j\right)
\frac{\pi n^{1/2}N^{1/2}}{q}\left(
-\frac{m}{N}\zeta^{-1/2-j}+\frac{1}{2j-1}\zeta^{1/2-j}\right).
\ena
Thus the stationary point is $\zeta_0=mN^{-1}$ and
$\varpi^{(j)}(\zeta)\ll_j n^{1/2}N^{1/2}/q$ for $j\geq 2$.
By \eqref{g rapid decay}, we have $\phi^{(j)}(\zeta)\ll_j N^{\varepsilon}$
(Here we note that for $C\leq Q^{1-\varepsilon}$ and $\Xi\ll N^{-\varepsilon}$,
we can replace $g\bigg(q,\frac{n^{1/2}Q}{\zeta^{1/2}N^{1/2}}\bigg)$ by 1 at the cost of a negligible error).
Applying Lemma \ref{lemma:exponentialintegral} (2) with $X=Z=1$ and
$Y=R=n^{1/2}N^{1/2}/q\gg N^{\varepsilon}$, we obtain
\bea\label{K-integral}
\mathcal{K}(m,n,q,\Xi)
=\frac{n^{1/4}q^{1/2}Q}{N^{3/4}}e\left(\frac{2\sqrt{mn}}{q}\right)
F\left(\frac{m}{N}\right)+O_A\left(N^{-A}\right),
\eea
where $F(x)=F(x;\Xi)$ is an inert function (depending on $A$ and $\Xi$)
supported on $x\asymp 1$.

Substituting \eqref{K-integral} into \eqref{after first Voronoi}, we obtain
\bea\label{after first Voronoi-1}
\mathscr{S}^*(C,\Xi)&=&N^{-1/4-it}\sum_{q\sim C}\frac{1}{q^{1/2}}\;
\sideset{}{^\star}\sum_{a\bmod{q}}\;
\sum_{n\asymp  N\Xi^2/Q^2}\frac{\lambda_g(n)}{n^{1/4}}
e\left(\frac{n\overline{a}}{q}\right)
\nonumber\\
&&\times\sum_{m\geq 1}\lambda_f(m)e\left(\frac{ma}{q}\right)\left(\frac{m}{N}\right)^{-it}
\widetilde{V}\left(\frac{m}{N}\right)e\left(\frac{2\sqrt{mn}}{q}\right)
+O_A\left(N^{-A}\right),
\eea
where $\widetilde{V}(x)=V(x)F(x)\in C_c^{\infty}(1,2)$ satisfying
$\widetilde{V}^{(j)}(x)\ll_j 1$ for any integer $j\geq 0$.

Now we apply Lemma \ref{GL2 Voronoi formula-version 2} with
$\varphi(x)=(x/N)^{-it}\widetilde{V}(x/N)e\left(2\sqrt{nx}/q\right)$ to the
$m$-sum in \eqref{after first Voronoi-1} and get
\bea\label{2nd Voronoi}
m\text{-sum}=q\sum_\pm\sum_{m\geq 1}\frac{\lambda_f(m)}{m}e\left(\pm\frac{\overline{a}m}{q}\right)
\Psi_{\varphi}^{\pm}\left(\frac{m}{q^2},n,q\right),
\eea
where by \eqref{integral 1}, for $\sigma>-1$,
\bea\label{integral after Voronois}
\Psi_{\varphi}^{\pm}(x,n,q)=\frac{N^{it}}{4\pi^2 i}\int_{(\sigma)}(\pi^2 x)^{-s}
\gamma_f^{\pm}(s)\bigg(
\int_0^{\infty}
\widetilde{V}\left(\frac{y}{N}\right)e\left(\frac{2n^{1/2}y^{1/2}}{q}\right)y^{-s-it-1}
\mathrm{d}y\bigg) \mathrm{d}s.
\eea

Further substituting \eqref{2nd Voronoi} into \eqref{after first Voronoi-1}
and writing the Ramanujan sum
$$S\left(m\pm n,0;q\right)=\sum_{d|(m\pm n,q)}d\mu(q/d),$$ then we have
\bea\label{after first Voronoi-2}
\mathscr{S}^*(C,\Xi)&=&N^{-1/4-it}\sum_\pm\sum_{q\sim C}q^{1/2}\sum_{d|q}d\mu\left(\frac{q}{d}\right)
\sum_{m\geq 1}\frac{\lambda_f(m)}{m}
\nonumber\\
&&\times\sum_{n\asymp  N\Xi^2/Q^2\atop n\equiv \pm m\bmod d}\frac{\lambda_g(n)}{n^{1/4}}
\Psi_{\varphi}^{\pm}\left(\frac{m}{q^2},n,q\right)
+O_A\left(N^{-A}\right).
\eea

We will show in Section \ref{proofs-of-technical-lemma} that
the integral $\Psi_{\varphi}^{\pm}(x,n,q)$ has the following properties.

\begin{lemma}\label{integral-lemma-1}
Let $B=2 n^{1/2}N^{1/2}/q.$

(1) If $T_2^{1-\varepsilon}\ll B\ll T_1^{1+\varepsilon}$, then
$\Psi_{\varphi}^{\pm}(x,n,q)=\mathbf{\Psi}_1+\mathbf{\Psi}_2$, where
$\mathbf{\Psi}_1$
is negligibly small unless $Nx\ll T_1^{1+\varepsilon}$, in which case
$$
\mathbf{\Psi}_1\ll (BNx)^{1/2},
$$
and $\mathbf{\Psi}_2$ is negligibly small unless $Nx\ll BT_1^{1+\varepsilon}$, in which case
$$
\mathbf{\Psi}_2\ll (Nx)^{1/2}.
$$

(2) If $B\ll T_2^{1-\varepsilon}$, then $\Psi_{\varphi}^{\pm}(x,n,q)$ is negligibly small
unless $x\asymp T_1|T_2|/N$, in which case
\bna
\Psi_{\varphi}^{\pm}(x,n,q)&=&(Nx)^{1/2+it}V_{\natural}^{\pm}(\tau_*)e\bigg(-\frac{T_1}{2\pi}\log\frac{T_1}{2e}
-\frac{T_2}{2\pi}\log\frac{|T_2|}{2e}+\frac{B\tau_0}{\pi}\\&&\qquad\qquad\qquad\qquad\quad
+\frac{B}{2\pi}\sum_{j=1}^{K}g_j\bigg(\frac{B}{T_1},\frac{B}{T_2}\bigg)\tau_0^{j+1}\bigg)
+ O_{A}(  N^{-A}),
\ena
where $K\geq 1$ is an integer, $A>0$ is a large constant depends on $K$,
$\tau_0=\left(T_1|T_2|/(4Nx)\right)^{1/2}$,
$V_{\natural}^{\pm}(\tau)$ is some inert function
supported on $\tau\asymp 1$, $\tau_*$ is defined in \eqref{stationary point-1}, and
$g_j(y_1,y_2)$ are some homogeneous polynomials of degree $j$ and satisfy
$g_j(y_1,y_2)\ll_j y_2^{j}$ for any integer $j\geq 1$.

(3) If $B\gg T_1^{1+\varepsilon}$, then $\Psi_{\varphi}^{\pm}(x,n,q)$
is negligibly small unless $Nx\asymp T_1|T_2|$, in which case
$
\Psi_{\varphi}^{\pm}(x,n,q)\ll (Nx)^{1/2}.$
\end{lemma}

Note that $B\asymp N\Xi/(CQ)\gg N^{\varepsilon}$
(see \eqref{zeta range} and the range of $n$ in \eqref{after first Voronoi-2}).
So by Lemma \ref{integral-lemma-1}, the properties of $\Psi_{\varphi}^{\pm}(x,n,q)$
depend on the size of $C$.  By Lemma \ref{integral-lemma-1}, we distinguish
two cases according to $C\geq N^{1+\varepsilon}\Xi/(Q|T_2|)$ or not.

\subsection{The case of large modulus}

In this section, we consider the case $C\geq N^{1+\varepsilon}\Xi/(Q|T_2|)$ which
is equivalent to the condition $B\ll T_2^{1-\varepsilon}$.
In this case, we use the second statement of Lemma \ref{integral-lemma-1}.
By \eqref{after first Voronoi-2} and Lemma 4.1 (2),
\bea\label{after first Voronoi-3}
\mathscr{S}^*(C,\Xi)&=&N^{1/4}e\bigg(-\frac{T_1}{2\pi}\log\frac{T_1}{2e}
-\frac{T_2}{2\pi}\log\frac{|T_2|}{2e}\bigg)\sum_\pm\sum_{q\sim C}\frac{1}{q^{1/2+2it}}
\sum_{d|q}d\mu\left(\frac{q}{d}\right)\nonumber\\&&\times
\sum_{m\asymp C^2T_1|T_2|/N}\frac{\lambda_f(m)}{m^{1/2-it}}
\sum_{n\asymp  N\Xi^2/Q^2\atop n\equiv \pm m\bmod d}\frac{\lambda_g(n)}{n^{1/4}}
\mathfrak{I}^{\pm}(m,n,q)
+O_A\left(N^{-A}\right),
\eea
where
\bea\label{I-definition}
\mathfrak{I}^{\pm}(m,n,q)=V_{\natural}^{\pm}(\tau_*)e\bigg(\frac{2\tau_0n^{1/2}N^{1/2}}{\pi q}
+\frac{B}{2\pi}\sum_{j=1}^{K}g_j\bigg(\frac{B}{T_1},\frac{B}{T_2}\bigg)\tau_0^{j+1}\bigg)
\eea
with
$
\tau_0=\left(T_1|T_2|/(4Nm)\right)^{1/2}q.
$

\subsubsection{Cauchy-Schwarz and Poisson summation}
Applying the Cauchy-Schwarz inequality to \eqref{after first Voronoi-3} and
using the Rankin-Selberg estimate
\eqref{GL2: Rankin Selberg}, one sees that
\bea\label{Cauchy}
\mathscr{S}^*(C,\Xi)&\ll&\frac{N^{3/4}}{CT_1^{1/2}|T_2|^{1/2}}\sum_\pm\sum_{q\sim C}
q^{-1/2}\sum_{d|q}d
\bigg(\sum_{m\asymp C^2T_1|T_2|/N }
|\lambda_f(m)|^2\bigg)^{1/2}\nonumber\\&&\bigg(\sum_{m\asymp  C^2T_1|T_2|/N}\bigg|
\sum_{n\asymp  N\Xi^2/Q^2\atop n\equiv \pm m\bmod d}\lambda_g(n)n^{-1/4}
\mathfrak{I}^{\pm}(m,n,q)\bigg|^2\bigg)^{1/2}\nonumber\\
&\ll&N^{1/4}\sum_\pm\sum_{q\sim C}
q^{-1/2}\sum_{d|q}d\sqrt{\mathbf{\Omega}(q,d)},
\eea
where
\bea\label{Omega}
\mathbf{\Omega}(q,d)=\sum_{m\in \mathbb{Z}}\omega\left(\frac{m}{C^2T_1|T_2|/N}\right)\bigg|
\sum_{n\asymp  N\Xi^2/Q^2\atop n\equiv \pm m\bmod d}\lambda_g(n)n^{-1/4}\mathfrak{I}^{\pm}(m,n,q)\bigg|^2.
\eea
Here $\omega$ is a nonnegative smooth function on $(0,+\infty)$, supported on $[2/3,3]$, and such
that $\omega(x)=1$ for $x\in [1,2]$.

Opening the absolute square, we break the $m$-sum into congruence classes
modulo $d$ and apply the Poisson summation formula to the sum over $m$ to get
\bea\label{omega-bound}
\mathbf{\Omega}(q,d)
&=&\sum_{n_1\asymp N\Xi^2/Q^2}\lambda_g(n_1)n_1^{-1/4}
\sum_{n_2\asymp N\Xi^2/Q^2\atop n_2\equiv n_1\bmod d}\overline{\lambda_g(n_2)}n_2^{-1/4}\nonumber\\&&\times
\sum_{m\equiv \pm n_1\bmod d}\omega\left(\frac{m}{C^2T_1|T_2|/N}\right)\mathfrak{I}^{\pm}(m,n_1,q)
\overline{\mathfrak{I}^{\pm}(m,n_2,q)}\nonumber\\
&=&\frac{C^2T_1|T_2|}{dN}\sum_{n_1\asymp N\Xi^2/Q^2}\lambda_g(n_1)n_1^{-1/4}
\sum_{n_2\asymp N\Xi^2/Q^2\atop n_2\equiv n_1\bmod d}\overline{\lambda_g(n_2)}n_2^{-1/4}
\nonumber\\&&\times
\sum_{\widetilde{m}\in \mathbb{Z}}e\left(-\frac{\widetilde{m}n_1}{d}\right)
\mathcal{H}\left(\frac{C^2T_1|T_2|\widetilde{m}}{dN}\right),
\eea
where the integral $\mathcal{H}(x)=\mathcal{H}(x;n_1,n_2,q)$ is given by
\bea\label{H-integral}
\mathcal{H}(x)=\int_{\mathbb{R}}
\omega\left(\xi\right)
\mathfrak{I}^{\pm}\left(C^2T_1|T_2|\xi/N,n_1,q\right)
\overline{\mathfrak{I}^{\pm}\left(C^2T_1|T_2|\xi/N,n_2,q\right)}
\, e\left(-x\xi\right)\mathrm{d}\xi.
\eea

We have the following estimates for $\mathcal{H}(x)$, whose proofs we postpone to
Section \ref{proofs-of-technical-lemma}.

\begin{lemma}\label{integral:lemma} Assume $C$ satisfies
$C\geq N^{1+\varepsilon}\Xi/(Q|T_2|)$ and $n_i\sim N\Xi^2/Q^2$, $i=1,2.$

(1) We have $\mathcal{H}(x)\ll 1$ for any $x\in \mathbb{R}$.

(2) For $x\gg N^{1+\varepsilon}\Xi/(CQ)$, we have $\mathcal{H}(x)\ll_A N^{-A}$.

(3) For $x\neq 0$, we have
$\mathcal{H}(x)\ll |x|^{-1/2}$.

(4) $\mathcal{H}(0)$
is negligibly small unless $|n_1-n_2|\ll N^{\varepsilon}$.

\end{lemma}

With estimates for $\mathcal{H}(x)$ ready, we now continue with the treatment
of $\mathbf{\Omega}(q,d)$
in \eqref{omega-bound}. by Lemma \ref{integral:lemma} (2), the contribution from the terms with
\bea\label{m range-final}
|\widetilde{m}|\gg \frac{dN^{2+\varepsilon}\Xi}{C^3QT_1|T_2|}:=N_1
\eea
to $\mathbf{\Omega}(q,d)$ is negligible. So
we only need to consider the range $0\leq |\widetilde{m}|\ll N_1$.

We treat the cases where $\widetilde{m}=0$ and $\tilde{m}\neq 0$ separately
and denote their contributions to $\mathbf{\Omega}(q,d)$ by
$\mathbf{\Omega}_0$ and $\mathbf{\Omega}_{\neq 0}$, respectively.

\subsubsection{The zero frequency}\label{The zero frequency}

Let $\mathbf{\Sigma}_{0}$ denote the contribution
of $\mathbf{\Omega}_0$ to \eqref{omega-bound}. Correspondingly, we denote its contribution to
\eqref{Cauchy} by $\mathbf{\Sigma}_0$.

\begin{lemma}\label{lemma:zero}
We have
\bna
\mathbf{\Sigma}_0
\ll  N^{\varepsilon}Q^{3/2}T_1^{1/2}|T_2|^{1/2}.
\ena
\end{lemma}

\begin{proof}
Splitting the sum over $n_1$ and $n_2$ according as $n_1=n_2$ or not,
and applying Lemma \ref{integral:lemma} (4),
the Rankin-Selberg estimate \eqref{GL2: Rankin Selberg}
and using the inequality
$|\lambda_g(n_1)\lambda_g(n_2)|\leq |\lambda_g(n_1)|^2+|\lambda_g(n_2)|^2$, we have
\bna
\mathbf{\Omega}_0
&\ll&
\frac{M}{dtN_1^{1/2}}
\mathop{\sum\sum}_{n_1,n_2\asymp N\Xi^2/Q^2\atop |n_1-n_2|\ll N^{\varepsilon}}
|\lambda_g(n_1)||\lambda_g(n_2)|
\\
&\ll&\frac{C^2T_1|T_2|}{dN}\frac{Q}{N^{1/2}\Xi}\sum_{n_1\asymp N\Xi^2/Q^2}|\lambda_g(n_1)|^2
\sum_{n_2\asymp N\Xi^2/Q^2\atop |n_1-n_2|\ll N^{\varepsilon}}1\\
&\ll&N^{\varepsilon}\frac{C^2T_1|T_2|\Xi}{dN^{1/2}Q}.
\ena
This bound when substituted in place of $\mathbf{\Omega}(q,d)$ into \eqref{Cauchy} yields that
\bna
\mathbf{\Sigma}_0\ll N^{1/4+\varepsilon}\sum_{\pm}\sum_{q\sim C}q^{-1/2}\sum_{d|q}d
\frac{CT_1^{1/2}|T_2|^{1/2}\Xi^{1/2}}{d^{1/2}N^{1/4}Q^{1/2}}
\ll N^{\varepsilon}Q^{3/2}T_1^{1/2}|T_2|^{1/2}.
\ena
This proves the lemma.

\end{proof}

\subsubsection{The non-zero frequencies}\label{The non-zero frequencies}

Recall $\mathbf{\Omega}_{\neq 0}$ denotes the contribution from the terms with $\tilde{m}\neq 0$
to $\mathbf{\Omega}(q,d)$ in
\eqref{omega-bound}. Correspondingly, we denote its contribution to
\eqref{Cauchy} by $\mathbf{\Sigma}_{\neq 0}$. Using the inequality
$|\lambda_g(n_1)\lambda_g(n_2)|\leq |\lambda_g(n_1)|^2+|\lambda_g(n_2)|^2$, we have
\bea\label{bound-med}
\mathbf{\Omega}_{\neq 0}\ll
\frac{C^2QT_1|T_2|}{dN^{3/2}\Xi}\sum_{n_1\asymp N\Xi^2/Q^2}|\lambda_g(n_1)|^2
\sum_{n_2\asymp N\Xi^2/Q^2\atop n_2\equiv n_1\bmod d}\;
\sum_{0\neq \widetilde{m}\ll N_1}\bigg|\mathcal{H}\left(\frac{C^2T_1|T_2|\widetilde{m}}{dN}\right)\bigg|,
\eea
where $N_1$ is defined in \eqref{m range-final}.

\begin{lemma}\label{lemma:nonzero}
Assume
\bea\label{assumption: range 2}
Q<N^{1/3}.
\eea
We have
\bna
\mathbf{\Sigma}_{\neq 0}\ll N^{5/4+\varepsilon}/Q.
\ena
\end{lemma}
\begin{proof}
For $x=C^2T_1|T_2|\widetilde{m}/(dN)$,
we split the sum over $\tilde{m}$ according to
$x\ll N^{\varepsilon}$ or not.
Set
\bea\label{N3}
N_2:=\frac{dN^{1+\varepsilon}}{C^2T_1|T_2|}.
\eea
For
$0\neq \tilde{m}\ll N_2$, we use the bound
$\mathcal{H}(x)\ll 1$ in Lemma \ref{integral:lemma}
(1), and for the remaining part we apply the bound
$\mathcal{H}(x)\ll |x|^{-1/2}$ in Lemma \ref{integral:lemma}
(3). By \eqref{bound-med}, we have
\bna
\mathbf{\Omega}_{\neq 0}&\ll&\frac{C^2QT_1|T_2|}{dN^{3/2}\Xi}\sum_{n_1\asymp N\Xi^2/Q^2}|\lambda_g(n_1)|^2
\sum_{n_2\asymp N\Xi^2/Q^2\atop n_2\equiv n_1\bmod d}\;
\sum_{0\neq \widetilde{m}\ll N_2}1\\
&&+\frac{C^2QT_1|T_2|}{dN^{3/2}\Xi}\sum_{n_1\asymp N\Xi^2/Q^2}|\lambda_g(n_1)|^2
\sum_{n_2\asymp N\Xi^2/Q^2\atop n_2\equiv n_1\bmod d}\;
\sum_{N_2\ll \widetilde{m}\ll N_1}\left(\frac{C^2T_1|T_2|\widetilde{m}}{dN}\right)^{-1/2}\\
  &\ll&\frac{C^2T_1|T_2|\Xi N_2}{dQ N^{1/2}}\left(1+\frac{N\Xi^2}{dQ^2}\right)
+\frac{CT_1^{1/2}|T_2|^{1/2}|\Xi N_1^{1/2}}{d^{1/2}Q}\left(1+\frac{N\Xi^2}{dQ^2}\right)\nonumber\\
    &\ll&\frac{CT_1^{1/2}|T_2|^{1/2}\Xi }{d^{1/2}Q}\left(1+\frac{N\Xi^2}{dQ^2}\right)
    \left(\frac{CT_1^{1/2}|T_2|^{1/2}| N_2}{d^{1/2}N^{1/2}}+N_1^{1/2}\right).
\ena
Here we have applied \eqref{Rankin-Selberg}.
By \eqref{m range-final} and \eqref{N3},
\bna
\mathbf{\Omega}_{\neq 0}
&\ll&\frac{CT_1^{1/2}|T_2|^{1/2}\Xi }{Q}\left(1+\frac{N\Xi^2}{dQ^2}\right)
\left(\frac{N^{1/2}}{CT_1^{1/2}|T_2|^{1/2}}+\frac{N\Xi^{1/2}}{C^{3/2}Q^{1/2}T_1^{1/2}|T_2|^{1/2}}\right)\\
&\ll&\frac{N^{1+\varepsilon}}{C^{1/2}Q^{3/2}}\left(1+\frac{N}{dQ^2}\right)
\ena
since $\Xi\ll N^{\varepsilon}$ and $Q<N^{1/2-\varepsilon}$ by \eqref{assumption 1}.
This bound when substituted in place of $\mathbf{\Omega}(q,d)$ in \eqref{Cauchy}
gives that
\bna
\mathbf{\Sigma}_{\neq 0}
&\ll&N^{1/4}\sum_{q\sim C}q^{-1/2}\sum_{d|q}
\frac{dN^{1/2+\varepsilon}}{C^{1/4}Q^{3/4}}\left(1+\frac{N^{1/2}}{d^{1/2}Q}\right)\\
&\ll&N^{3/4}Q^{-3/4}C^{3/4}\left(C^{1/2}+N^{1/2}/Q\right)\\
&\ll& N^{3/4+\varepsilon}\left(Q^{1/2}+N^{1/2}/Q\right)\\
&\ll&N^{5/4+\varepsilon}/Q
\ena
provided that $Q<N^{1/3}$.

\end{proof}

\subsection{The case of small modulus}

In this section, we deal with the case $1\ll C\leq N^{1+\varepsilon}\Xi/(Q|T_2|)$
which is equivalent to the condition $B\gg T_2^{1-\varepsilon}$ (see Lemma \ref{integral-lemma-1}).
In this case, we will use the first and third statement of Lemma \eqref{after first Voronoi-2}.
By \eqref{after first Voronoi-2}, \eqref{individual bound} and (1) and (3) of Lemma \eqref{after first Voronoi-2},
\bea\label{small cases}
\mathscr{S}^*(C,\Xi)&\ll&\frac{C^{1/2}}{N^{1/4}}\left(\frac{N\Xi^2}{Q^2}\right)^{\vartheta-1/4}\sum_\pm\sum_{q\sim C}\sum_{d|q}d
\sum_{m\geq 1}\frac{|\lambda_f(m)|}{m}
\sum_{n\asymp  N\Xi^2/Q^2\atop n\equiv \pm m\bmod d}
\left|\Psi_{\varphi}^{\pm}\left(\frac{m}{q^2},n,q\right)\right|\\
&\ll&\mathbf{R}_1+\mathbf{R}_2+\mathbf{1}_{C\leq N\Xi/(QT_1^{1+\varepsilon})}\mathbf{R}_3,
\eea
where $\mathbf{1}_{S}=1$ is $S$ is true and equals 0 otherwise,
\bna
\mathbf{R}_1&=&\frac{C^{1/2}}{N^{1/4}}\left(\frac{N\Xi^2}{Q^2}\right)^{\vartheta-1/4}\sum_{q\sim C}\sum_{d|q}d
\sum_{m\ll C^2T_1^{1+\varepsilon}/N}\frac{|\lambda_f(m)|}{m}
\sum_{n\asymp  N\Xi^2/Q^2\atop n\equiv \pm m\bmod d}
\left(\frac{BNm}{q^2}\right)^{1/2},\\
\mathbf{R}_2&=&\frac{C^{1/2}}{N^{1/4}}\left(\frac{N\Xi^2}{Q^2}\right)^{\vartheta-1/4}\sum_{q\sim C}\sum_{d|q}d
\sum_{m\ll BC^2T_1^{1+\varepsilon}/N}\frac{|\lambda_f(m)|}{m}
\sum_{n\asymp  N\Xi^2/Q^2\atop n\equiv \pm m\bmod d}
\left(\frac{Nm}{q^2}\right)^{1/2},
\ena
and
\bna
\mathbf{R}_3=\frac{C^{1/2}}{N^{1/4}}\left(\frac{N\Xi^2}{Q^2}\right)^{\vartheta-1/4}\sum_{q\sim C}\sum_{d|q}d
\sum_{m\asymp C^2T_1|T_2|/N}\frac{|\lambda_f(m)|}{m}
\sum_{n\asymp  N\Xi^2/Q^2\atop n\equiv \pm m\bmod d}
\left(\frac{Nm}{q^2}\right)^{1/2}.
\ena
Recall that $B\asymp N\Xi/(CQ)$, $\Xi\ll N^{\varepsilon}$ and $C\leq N^{1+\varepsilon}\Xi/(Q|T_2|)$.
By \eqref{Rankin-Selberg}, we have
\bna
\mathbf{R}_1&\ll&\frac{B^{1/2}N^{1/4}}{C^{1/2}}\left(\frac{N\Xi^2}{Q^2}\right)^{\vartheta-1/4}
\sum_{q\sim C}\sum_{d|q}d
\left(\frac{C^2T_1^{1+\varepsilon}}{N}\right)^{1/2}
\left(1+\frac{N\Xi^2}{dQ^2}\right)\\
&\ll&T_1^{1/2+\varepsilon}\left(\frac{N\Xi^2}{Q^2}\right)^{\vartheta}C
\left(C+\frac{N\Xi^2}{Q^2}\right)\\
&\ll&T_1^{1/2+\varepsilon}\left(\frac{N\Xi^2}{Q^2}\right)^{\vartheta}
\frac{N\Xi}{Q|T_2|}
\left(\frac{N\Xi}{Q|T_2|}+\frac{N\Xi^2}{Q^2}\right)\\
&\ll&T_1^{1/2+\varepsilon}\left(\frac{N^{1+\varepsilon}}{Q^2}\right)^{\vartheta+1}
\frac{N^{1+\varepsilon}}{Q|T_2|}
\ena
assuming
\bea\label{assumption 3}
Q<|T_2|.
\eea
Thus
\bea\label{R-1}
\mathbf{R}_1\ll \frac{N^{2+\vartheta+\varepsilon}T_1^{1/2}}{Q^{3+2\vartheta}|T_2|}.
\eea
Similarly,
\bea\label{R-2}
\mathbf{R}_2&\ll&\frac{N^{1/4}}{C^{1/2}}\left(\frac{N\Xi^2}{Q^2}\right)^{\vartheta-1/4}
\sum_{q\sim C}\sum_{d|q}d
\left(\frac{BC^2T_1^{1+\varepsilon}}{N}\right)^{1/2}
\left(1+\frac{N\Xi^2}{dQ^2}\right)\nonumber\\
&\ll&\frac{N^{2+\vartheta+\varepsilon}T_1^{1/2}}{Q^{3+2\vartheta}|T_2|}
\eea
and for $C\leq N\Xi/(QT_1^{1+\varepsilon})$, we have
\bea\label{R-3}
\mathbf{R}_3&\ll&\frac{N^{1/4}}{C^{1/2}}\left(\frac{N\Xi^2}{Q^2}\right)^{\vartheta-1/4}
\sum_{q\sim C}\sum_{d|q}d
\left(\frac{C^2T_1|T_2|}{N}\right)^{1/2}
\left(1+\frac{N\Xi^2}{dQ^2}\right)\nonumber\\
&\ll&\frac{T_1^{1/2+\varepsilon}|T_2|^{1/2}Q^{1/2}}{N^{1/2}\Xi^{1/2}}
\left(\frac{N\Xi^2}{Q^2}\right)^{\vartheta}C^{3/2}
\left(C+\frac{N\Xi^2}{Q^2}\right)\nonumber\\
&\ll&\frac{T_1^{1/2+\varepsilon}|T_2|^{1/2}Q^{1/2}}{N^{1/2}\Xi^{1/2}}\left(\frac{N\Xi^2}{Q^2}\right)^{\vartheta}
\left(\frac{N\Xi}{QT_1^{1+\varepsilon}}\right)^{3/2}
\left(\frac{N\Xi}{QT_1^{1+\varepsilon}}+\frac{N\Xi^2}{Q^2}\right)\nonumber\\
&\ll&\frac{N^{2+\vartheta+\varepsilon}T_1^{1/2}}{Q^{3+2\vartheta}|T_2|}
\eea
under the assumption in \eqref{assumption 3}.

By \eqref{small cases} and \eqref{R-1}-\eqref{R-3}, we conclude that for $1\ll C\leq N^{1+\varepsilon}\Xi/(Q|T_2|)$,
\bea\label{small contribution}
\mathscr{S}^*(C,\Xi)\ll \frac{N^{2+\vartheta+\varepsilon}T_1^{1/2}}{Q^{3+2\vartheta}|T_2|}.
\eea

\subsection{Conclusion}
By inserting the upper bounds in Lemmas \ref{lemma:zero}
and \ref{lemma:nonzero} into \eqref{Cauchy}, we have
\bna
\mathscr{S}^*(C,\Xi)\ll N^{\varepsilon}
\left(Q^{3/2}T_1^{1/2}|T_2|^{1/2}+N^{5/4}/Q\right),
\ena
under the assumption $N^{1+\varepsilon}\Xi/(Qt) \leq C\ll Q$ and
\bea
\label{final-assumption-Q}
Q<N^{1/3}
\eea
which is a combination of \eqref{assumption 1} and
\eqref{assumption: range 2}. We set $Q=N^{1/2}/(T_1|T_2|)^{1/5}$ to balance the contribution. Then
this $Q$ also satisfies \eqref{assumption 3} and
for $N^{1+\varepsilon}\Xi/(Qt) \leq C\ll Q$,
\bea\label{large estimate}
\mathscr{S}^*(C,\Xi)\ll N^{3/4+\varepsilon}(T_1|T_2|)^{1/5}\ll  N^{3/4+\varepsilon}T_1^{2/5}
\eea
provided $N<(T_1|T_2|)^{6/5}$, which is satisfactory since we only need this estimate
in the range $N<T_1^{1+\varepsilon}|T_2|$. Moreover, for this choice of $Q$, when $C\leq X^{1+\varepsilon}\Xi/(Qt)$,
by \eqref{small contribution}, $\mathscr{S}^*(C,\Xi)$ is bounded by
\bea\label{small estimate}
\frac{N^{2+\vartheta+\varepsilon}T_1^{1/2}}{Q^{3+2\vartheta}|T_2|}=
N^{1/2+\varepsilon}T_1^{11/10+2\vartheta/5}|T_2|^{-2/5+2\vartheta/5}\ll
N^{1/2+\varepsilon}T_1^{39/50+2\vartheta/25}
\eea
by the condition $|T_2|\gg T_1^{4/5-\varepsilon}$ in \eqref{T2 condition}.
Substituting the estimates in \eqref{large estimate} and \eqref{small estimate}
for $\mathscr{S}^*(C,\Xi)$ into \eqref{C range}, we obtain
\bna
\mathcal{S}(N)\ll N^{3/4+\varepsilon}T_1^{2/5}+
N^{1/2+\varepsilon}T_1^{39/50+2\vartheta/25}.
\ena
Then by \eqref{step 1},
\bna
L\left(\frac{1}{2}+it,f\otimes g\right)&\ll&  T_1^{\varepsilon}
\sup_{ T_1^{9/5}\ll N\ll T_1^{1+\varepsilon}|T_2|}N^{1/4+\varepsilon}T_1^{2/5}+
T_1^{39/50+2\vartheta/25}+ T_1^{9/10}\\
&\ll&T_1^{9/10+\varepsilon}+T_1^{39/50+2\vartheta/25+\varepsilon}.
\ena
Note that the second term is dominated by
the first term since
we can take $\vartheta=7/64$ by \cite{KS}.
This completes the proof of Theorem \ref{subconvexity}.

\section{Estimation of integrals}\label{proofs-of-technical-lemma}

We first prove Lemma  \ref{integral-lemma-1}.
\begin{proof}[Proof of Lemma \ref{integral-lemma-1}]
The proof is similar as Huang \cite{HB2}.
Let $s=\sigma+i\tau$.
Making changes of variables $\tau\rightarrow \tau-t$ and $y\rightarrow Ny^2$ in \eqref{integral after Voronois}, one has
\bna
\Psi_{\varphi}^{\pm}(x,n,q)&=&\frac{1}{2\pi^2}\int_{\mathbb{R}}(\pi^2 Nx)^{-\sigma-i\tau+it}
\gamma_f^{\pm}(\sigma+i\tau-it)\varphi_{\sigma}(n,q,\tau)\mathrm{d}\tau,
\ena
where
\bna
\varphi_{\sigma}(n,q,\tau)=
\int_0^{\infty}
\widetilde{V}\left(y^2\right)y^{-2\sigma-1}\exp\left(i\varrho(y)\right)
\mathrm{d}y
\ena
with $\varrho(y)=2\pi By-2\tau\log y$ and $B=2 n^{1/2}N^{1/2}/q.$
Note that
\bna
\varrho'(y)&=&2\pi B-2\tau/y,\\
\varrho^{(j)}(y)
&=&2\tau (-1)^j(j-1)! y^{-j}\asymp |\tau|, \qquad j=2,3,\ldots.
\ena
By repeated integration by parts one shows that $\varphi_{\sigma}(n,q,\tau)$
is negligibly small unless $\tau>0$ and $\tau \asymp B$.
The stationary point is $y_0=\tau/(\pi B)$.
Recall that $B\asymp N\Xi/(CQ)\gg N^{\varepsilon}$.
Applying Lemma \ref{lemma:exponentialintegral} (2) with $X=Z=1$ and
$Y=R=\tau\gg N^{\varepsilon}$, we obtain
\bna
\varphi_{\sigma}(n,q,\tau)
=\tau^{-1/2}
V^{\natural}_{\sigma}\left(\frac{\tau}{\pi B}\right)
e\left(-\frac{\tau}{\pi}\log\frac{\tau}{\pi e B}\right)+O_A\left(N^{-A}\right),
\ena
where $V_{\sigma}^{\natural}(x)$ is an inert function (depending on $A$ and $\sigma$)
supported on $x\asymp 1$. Assembling these results, we obtain
\bna
\Psi_{\varphi}^{\pm}(x,n,q)&=&\frac{1}{2\pi^2}\int_0^{\infty}
(\pi^2 Nx)^{-\sigma-i\tau+it}\gamma_f^{\pm}\left(\sigma+i\tau-it\right)\nonumber\\&&\times
\tau^{-1/2}V^{\natural}_{\sigma}\left(\frac{\tau}{\pi B}\right)
e\left(-\frac{\tau}{\pi}\log\frac{\tau}{\pi e B}\right)\mathrm{d}\tau+O_A\left(N^{-A}\right).
\ena
Making a change of variable $\tau\rightarrow B\tau$,
\bea\label{The integral:1}
\Psi_{\varphi}^{\pm}(x,n,q)&=&\frac{B^{1/2}}{2\pi^2}\int_0^{\infty}
(\pi^2 Nx)^{-\sigma-iB\tau+it}\gamma_f^{\pm}\left(\sigma+iB\tau-it\right)\nonumber\\&&\times
\tau^{-1/2}V^{\natural}_{\sigma}\left(\frac{\tau}{\pi}\right)
e\left(-\frac{B\tau}{\pi}\log\frac{\tau}{\pi e}\right)\mathrm{d}\tau+O_A\left(N^{-A}\right),
\eea
where by \eqref{gamma},
\bea\label{Gamma3}
\gamma_f^{\pm}\left(\sigma+iB\tau-it\right)=
\prod_{j=1,2}\frac{\Gamma\left(\frac{1+\sigma+i(B\tau-T_j)}{2}\right)}
{\Gamma\left(\frac{-\sigma-i(B\tau-T_j)}{2}\right)}\pm
\prod_{j=1,2}\frac{\Gamma\left(\frac{2+\sigma+i(B\tau-T_j)}{2}\right)}
{\Gamma\left(\frac{1-\sigma-i(B\tau-T_j)}{2}\right)}.
\eea

(1) For $T_2^{1-\varepsilon}<B<T_1^{1+\varepsilon}$, we divide the range of
$\tau$ into two pieces: $$(0,\infty)=\{\tau||T_1-B\tau|\leq T_1^{\varepsilon}\,
\text{or}\,|T_2-B\tau|\leq T_1^{\varepsilon}\}
\cup \{\tau||T_1-B\tau|>T_1^{\varepsilon}, |T_2-B\tau|>T_1^{\varepsilon}\}:=\mathbf{I}_1+\mathbf{I}_2$$
and correspondingly denote by the integral over $\mathbf{I}_j$ by
$\mathbf{\Psi}_j$, $j=1,2$. Then by \eqref{Gamma-2},
\bna
\mathbf{\Psi}_1&\ll&B^{1/2}(Nx)^{-\sigma}\int_{\mathbf{I}_1}
(|T_1-B\tau||T_2-B\tau|)^{\sigma+1/2}
\left|V^{\natural}_{\sigma}\left(\frac{\tau}{\pi}\right)\right|\mathrm{d}\tau\\
&\ll&B^{1/2}T_1^{1/2+\varepsilon}(Nx/T_1^{1+\varepsilon})^{-\sigma}.
\ena
By taking $\sigma$ sufficiently large, one sees that $\Psi_1$
is negligibly small unless $Nx\ll T_1^{1+\varepsilon}$, in which case by taking $\sigma=-1/2$ we have
the estimate
\bea\label{p-1}
\mathbf{\Psi}_1\ll (BNx)^{1/2}.
\eea

For $\tau\in \mathbf{I}_2$, by \eqref{Gamma3} and Stirling's approximation
in \eqref{Stirling approximation},
we have
\bea\label{Gamma4}
&&\gamma_f^{\pm}\left(\sigma+iB\tau-it\right)=\left(\prod_{j=1,2}
\left(\frac{|T_j-B\tau|}{2e}\right)^{i(B\tau-T_j)}
|T_j-B\tau|^{\sigma+1/2}\right)\nonumber
\\&&\times
\big(h_{\sigma,1}(B\tau-T_1)h_{\sigma,1}(B\tau-T_2)\pm h_{\sigma,2}(B\tau-T_1)h_{\sigma,2}(B\tau-T_2)\big)+
O_{\sigma,K_3}\big(T_1^{-K_3}\big).
\eea
Then by \eqref{The integral:1} and \eqref{Gamma4},
\bea\label{The integral:20}
\mathbf{\Psi}_2&=&\frac{B^{1/2}}{2\pi^2}\int_0^{\infty}
(\pi^2 Nx)^{-\sigma-iB\tau+it}\left(\prod_{j=1,2}
\left(\frac{|T_j-B\tau|}{2e}\right)^{i(B\tau-T_j)}
|T_j-B\tau|^{\sigma+1/2}\right)\nonumber
\\&&\times
\big(h_{\sigma,1}(B\tau-T_1)h_{\sigma,1}(B\tau-T_2)\pm h_{\sigma,2}(B\tau-T_1)h_{\sigma,2}(B\tau-T_2)\big)\nonumber\\&&\times
\tau^{-1/2}V^{\natural}_{\sigma}\left(\frac{\tau}{\pi}\right)
e\left(-\frac{B\tau}{\pi}\log\frac{\tau}{\pi e}\right)\mathrm{d}\tau+\mathbf{\Psi}_3,
\eea
where $h_{\sigma,j}(x)$, $j=1,2$, satisfy $h_{\sigma,j}(x)\ll_{\sigma,j,K_4} 1$ and
$
x^{\ell}h_{\sigma,j}^{(\ell)}(x)\ll_{\sigma,j,\ell,K_4} x^{-1}
$
for any integer $\ell\geq 1$, and
\bna
\mathbf{\Psi}_3\ll B^{1/2}(Nx)^{-\sigma}\int_{\mathbf{I}_1}(|T_1-B\tau||T_2-B\tau|)^{\sigma+1/2}\left|V^{\natural}_{\sigma}
\left(\frac{\tau}{\pi}\right)\right|\mathrm{d}\tau
\ll B^{1/2}T_1^{1/2+\varepsilon}(Nx/T_1^{1+\varepsilon})^{-\sigma}
\ena
which can be negligibly small unless $Nx\ll T_1^{1+\varepsilon}$,
in which case by taking $\sigma=-1/2$ we have
\bea\label{p-2}
\mathbf{\Psi}_3\ll (BNx)^{1/2}.
\eea
Denote the first term in \eqref{The integral:20} by $\mathbf{\Psi}_2^0$. Then
\bna
\mathbf{\Psi}_2^0\ll B^{1/2}(Nx)^{-\sigma}\int_{\tau\asymp 1}(|T_1-B\tau||T_2-B\tau|)^{\sigma+1/2}
\mathrm{d}\tau
\ll BT_1^{1/2+\varepsilon}\left(\frac{Nx}{BT_1^{1+\varepsilon}}\right)^{-\sigma}
\ena
which can be negligibly small unless $Nx\ll BT_1^{1+\varepsilon}$,
in which case by taking $\sigma=-1/2$,
\bna
\mathbf{\Psi}_2^0=B^{1/2}(Nx)^{1/2+it}\int_0^{\infty}G(\tau)
\exp\left(i\eta(\tau)\right)\mathrm{d}\tau,
\ena
where, temporarily,
\bna
G(\tau)=\frac{\pi^{2it-1}}{2\sqrt{\tau}}V^{\natural}_{\sigma}\left(\frac{\tau}{\pi}\right)\big(h_{\sigma,1}(B\tau-T_1)h_{\sigma,1}(B\tau-T_2)
\pm h_{\sigma,2}(B\tau-T_1)h_{\sigma,2}(B\tau-T_2)\big)
\ena
and
\bna
\eta(\tau)=-B\tau\log \frac{Nx}{e^2}+(B\tau-T_1)\log\frac{|T_1-B\tau|}{2e}
+(B\tau-T_2)\log\frac{|T_2-B\tau|}{2e}-2B\tau\log \tau.
\ena
Note that
\bna
\eta'(\tau)&=&B\log\frac{|T_1-B\tau||T_2-B\tau|}{4Nx\tau^2},\\
\eta''(\tau)&=&B\left(\frac{1}{\tau-T_1/B}
+\frac{1}{\tau-T_2/B}-\frac{2}{\tau}\right)
\ena
and
\bna
\int_{\min\{|B\tau-T_1|,|B\tau-T_2|\}>\sqrt{B} }
\left|\frac{\mathrm{d}G(\tau)}{\mathrm{d}\tau}\right|\mathrm{d}\tau
\ll \max_{\tau\asymp 1}\left\{
\frac{B}{|B\tau-T_1|^2}
,\frac{B}{|B\tau-T_2|^2},1\right\}\ll 1.
\ena
Moreover, for $\min_{\tau\asymp 1}\{|B\tau-T_1|,|B\tau-T_2|\}>\sqrt{B}$,
\bna
\eta''(\tau)\asymp \left\{\begin{array}{lr}B\max\limits_{\tau\asymp 1}
|\tau-T_1/B|^{-1},&\mathrm{if}\, T_1^{1-\varepsilon}\ll B\ll T_1^{1+\varepsilon}\\
B,&\mathrm{if}\, T_2^{1+\varepsilon}\ll B\ll T_1^{1-\varepsilon},\\
B\max\limits_{\tau\asymp 1}|\tau-T_2/B|^{-1},
&\mathrm{if}\, T_2^{1-\varepsilon}\ll B\ll T_2^{1+\varepsilon}.
\end{array}\right.
\ena
Then by Lemma \ref{lem: 2st derivative test, dim 1},
\bna
&&B^{1/2}(Nx)^{1/2+it}\int_{\min\{|B\tau-T_1|,|B\tau-T_2|\}>\sqrt{B} }G(\tau)
\exp\left(i\eta(\tau)\right)\mathrm{d}\tau\\
&\ll& (Nx)^{1/2}\left\{\begin{array}{lr}\min\limits_{|B\tau-T_1|>\sqrt{B},\tau\asymp 1}
|\tau-T_1/B|^{1/2},&\mathrm{if}\, T_1^{1-\varepsilon}\ll B\ll T_1^{1+\varepsilon}\\&\\
1,&\mathrm{if}\, T_2^{1+\varepsilon}\ll B\ll T_1^{1-\varepsilon},\\&\\
\min\limits_{|B\tau-T_2|>\sqrt{B},\tau\asymp 1}|\tau-T_2/B|^{1/2},
&\mathrm{if}\, T_2^{1-\varepsilon}\ll B\ll T_2^{1+\varepsilon}.
\end{array}\right.\\
&\ll&(Nx)^{1/2}.
\ena
Trivially, we have
\bna
B^{1/2}(Nx)^{1/2+it}\int_{\min\{|B\tau-T_1|,|B\tau-T_2|\}\leq \sqrt{B} }G(\tau)
\exp\left(i\eta(\tau)\right)\mathrm{d}\tau\ll (Nx)^{1/2}.
\ena
Assembling the above results, we conclude that
\bea\label{p-3}
\mathbf{\Psi}_2^0\ll (Nx)^{1/2}.
\eea
Then the first statement follows from \eqref{p-1},\eqref{The integral:20}, \eqref{p-2} and \eqref{p-3}.

(2) For the second statement in  Lemma  \ref{integral-lemma-1}, we
take $\sigma=-1/2$ in \eqref{The integral:1}, we have
\bea\label{The integral:2}
\Psi_{\varphi}^{\pm}(x,n,q)&=&\frac{B^{1/2}}{2\pi^2}\int_0^{\infty}
(\pi^2 Nx)^{1/2-iB\tau+it}\gamma_f^{\pm}\left(-\frac{1}{2}+iB\tau-it\right)\nonumber\\&&\times
\tau^{-1/2}V^{\natural}\left(\frac{\tau}{\pi}\right)
e\left(-\frac{B\tau}{\pi}\log\frac{\tau}{\pi e}\right)\mathrm{d}\tau+O_A\left(N^{-A}\right).
\eea
where $V^{\natural}(x)=V^{\natural}_{-1/2}(x)$ and by \eqref{Gamma3},
\bna
\gamma_f^{\pm}\left(-\frac{1}{2}+iB\tau-it\right)=
\prod_{j=1,2}\frac{\Gamma\left(\frac{1/2+i(B\tau-T_j)}{2}\right)}
{\Gamma\left(\frac{1/2-i(B\tau-T_j)}{2}\right)}\pm
\prod_{j=1,2}\frac{\Gamma\left(\frac{3/2+i(B\tau-T_j)}{2}\right)}
{\Gamma\left(\frac{3/2-i(B\tau-T_j)}{2}\right)}.
\ena
Since $B\ll T_2^{1-\varepsilon}$, using Stirling's approximation
in \eqref{Stirling approximation}, we derive
\bea\label{gamma-2}
\begin{split}
&\gamma_f^{\pm}\left(-\frac{1}{2}+iB\tau-it\right)
=\left(\frac{T_1-B\tau}{2e}\right)^{i(B\tau-T_1)}
\left(\frac{|T_2-B\tau|}{2e}\right)^{i(B\tau-T_2)}\\&\qquad\times
\big(h_{1}(B\tau-T_1)h_1(B\tau-T_2)\pm h_2(B\tau-T_1)h_2(B\tau-T_2)\big)+
O_{K_4}\big(T_1^{-K_4}\big),
\end{split}
\eea
where $h_j(x)$, $j=1,2$, satisfying $h_j(x)\ll_j1$ and
$
x^{\ell}h_{j}^{(\ell)}(x)\ll_{j,\ell,K_4} x^{-1}
$
for any integer $\ell\geq 1$. Plugging \eqref{gamma-2} into \eqref{The integral:2}, one has
\bea\label{The integral:3}
\Psi_{\varphi}^{\pm}(x,n,q)=B^{1/2}(Nx)^{1/2+it}\int_0^{\infty}
V_0^{\pm}(\tau)\exp\left(i\varrho_0(\tau)\right)\mathrm{d}\tau+O_A\left(N^{-A}\right),
\eea
where
\bna
V_0^{\pm}(\tau)=\frac{\pi^{2it-1}}{2\sqrt{\tau}}V^{\natural}\left(\frac{\tau}{\pi}\right)
\big(h_{1}(B\tau-T_1)h_1(B\tau-T_2)\pm h_2(B\tau-T_1)h_2(B\tau-T_2)\big)
\ena
satisfying $\mathrm{d}^{\ell}V_0^{\pm}(\tau)/\mathrm{d}\tau^{\ell}\ll_{\ell} 1$ for any integer $\ell\geq 0$, and
\bea\label{varrho-definition}
\varrho_0(\tau)=-B\tau\log \frac{Nx}{e^2}+(B\tau-T_1)\log\frac{T_1-B\tau}{2e}
+(B\tau-T_2)\log\frac{|T_2-B\tau|}{2e}-2B\tau\log \tau.
\eea
We compute
\bna
\varrho'_0(\tau)&=&B\log\frac{(T_1-B\tau)|T_2-B\tau|}{4Nx\tau^2},\\
\varrho_0^{(j)}(\tau)&=&B(-1)^j(j-2)!\left(\frac{1}{(\tau-T_1/B)^{j-1}}
+\frac{1}{(\tau-T_2/B)^{j-1}}-\frac{2}{\tau^{j-1}}\right)\asymp B,\quad j=2,3,\ldots.
\ena
By repeated integration by parts one shows that $\Psi_{\varphi}^{\pm}(x,n,q)$
is negligibly small unless $Nx\asymp T_1|T_2|$.
Denote $C_{\alpha}^j=\alpha(\alpha-1)\cdot\cdot\cdot (\alpha-j+1)/j!$.
By an iterative argument, the stationary point $\tau_*$ which is the solution to
the equation $\varrho'_0(\tau)=0$, i.e.,
$4Nx\tau^2=T_1|T_2|(1-T_1^{-1}B\tau)(1-T_2^{-1}B\tau)$, can be written as
\bea\label{stationary point-1}
\tau_*&=&\left(\frac{T_1|T_2|}{4Nx}\right)^{1/2}\left(1-\frac{B}{T_1}
\tau_*\right)^{1/2}\left(1-\frac{B}{T_2}
\tau_*\right)^{1/2}\nonumber\\
&=&\left(\frac{T_1|T_2|}{4Nx}\right)^{1/2} \bigg( \sum_{j=0}^{K_5}
C_{1/2}^{j}\left(\frac{-B}{T_1}\right)^j\tau_*^{j}
+O_{K_5}\left((B/T_1)^{K_5+1}\right) \bigg) \nonumber\\&&\times
\bigg( \sum_{j=0}^{K_6}
C_{1/2}^{j}\left(\frac{-B}{|T_2|}\right)^j\tau_*^{j}
+O_{K_6}\left((B/|T_2|)^{K_6+1}\right) \bigg) \nonumber\\
&=&\sum_{j=0}^{K}\tau_j+O_{K_7}\left((B/|T_2|)^{K_7+1}\right),
\eea
where
\bna
\tau_0&=&\left(\frac{T_1|T_2|}{4Nx}\right)^{1/2} \asymp 1,\\
\tau_1&=&\frac{-1}{2}\left(\frac{B}{T_1}+\frac{B}{|T_2|}\right)\tau_0^2\asymp \frac{B}{|T_2|},\\
\tau_2&=&\frac{-1}{8}\left(\frac{B^2}{T_1^2}-\frac{2B^2}{T_1|T_2|}+\frac{B^2}{|T_2|^2}\right)\tau_0^3\asymp \frac{B^2}{|T_2|^2},\\
\tau_{j}&=&f_j\left(\frac{B}{T_1},\frac{B}{|T_2|}\right)\tau_0^{j+1}
\asymp \left(\frac{B}{|T_2|}\right)^{j}, j=3,4\ldots,
\ena
for some homogeneous polynomials $f_j(y_1,y_2)$ of degree $j$.
Since $B\ll T_2^{1-\varepsilon}$, the
$O$-term in \eqref{stationary point-1} is $O(N^{-\varepsilon K_7 })$, which
can be arbitrarily small by taking $K_7$ sufficiently large.

Applying Lemma \ref{lemma:exponentialintegral} (2) with $X=Z=1$ and
$Y=R=\tau\gg N^{\varepsilon}$, we obtain
\bna
\int_0^{\infty}
V_0^{\pm}(\tau)\exp\left(i\varrho_0(\tau)\right)\mathrm{d}\tau=
B^{-1/2}V_{\natural}^{\pm}(\tau_*)e^{i\varrho_0(\tau_*)}+ O_{A}(  N^{-A}),
\ena
where $V_{\natural}^{\pm}(\tau)$ is some inert function
supported on $\tau\asymp 1$.
From \eqref{varrho-definition} and \eqref{stationary point-1} and
using Taylor series expansion $\log(1-y)=-\sum_{j=1}^{\infty}y^j/j, y\in (-1,1)$, we have
\bna
\varrho_0(\tau_*)&=&B\tau_*\log\frac{(T_1-B\tau_*)|T_2-B\tau_*|}{4Nx\tau_*^2}
-T_1\log\frac{T_1-B\tau_*}{2e}
-T_2\log\frac{|T_2-B\tau_*|}{2e}\\
&=&-T_1\log\frac{T_1-B\tau_*}{2e}
-T_2\log\frac{|T_2-B\tau_*|}{2e}\\
&=&-T_1\log\frac{T_1}{2e}
-T_2\log\frac{|T_2|}{2e}-T_1\log\left(1-BT_1^{-1}\tau_*\right)
-T_2\log\left(1-B T_2^{-1}\tau_*\right)\\
&=&-T_1\log\frac{T_1}{2e}
-T_2\log\frac{|T_2|}{2e}+B\left(\sum_{j=0}^{\infty}\frac{1}{j+1}\left(\left(B/T_1\right)^j+
\left(B/T_2\right)^j\right)\tau_*^{j+1}\right)\\
&=&-T_1\log\frac{T_1}{2e}
-T_2\log\frac{|T_2|}{2e}+B\sum_{j=0}^{K}g_j\left(\frac{B}{T_1},\frac{B}{T_2}\right)\tau_0^{j+1}
+O_{K}\left((B/|T_2|)^{K+1}\right)
\ena
for some homogeneous polynomials $g_j(y_1,y_2)$ of degree $j$ and satisfying
$g_j(y_1,y_2)\asymp y_2^{j}$ for any integer $j\geq 0$. In particular,
$g_0(y_1,y_2)=2$.
Hence,
\bea\label{The integral:4}
\int_0^{\infty}
V_0(\tau)\exp\left(i\varrho_0(\tau)\right)\mathrm{d}\tau&=&
B^{-1/2}V_{\natural}^{\pm}(\tau_*)e\left(-\frac{T_1}{2\pi}\log\frac{T_1}{2e}
-\frac{T_2}{2\pi}\log\frac{|T_2|}{2e}\right.\nonumber\\
&&\left.+\frac{B}{2\pi}\sum_{j=0}^{K}g_j\left(\frac{B}{T_1},\frac{B}{T_2}\right)\tau_0^{j+1}\right)
+ O_{A}(  N^{-A}),
\eea
where $A>0$ is a large constant depends on $K$. By \eqref{The integral:3} and \eqref{The integral:4},
\bna
\Psi_{\varphi}^{\pm}(x,n,q)&=&(Nx)^{1/2+it}V_{\natural}^{\pm}(\tau_*)e\bigg(-\frac{T_1}{2\pi}\log\frac{T_1}{2e}
-\frac{T_2}{2\pi}\log\frac{|T_2|}{2e}\\&&
+\frac{B}{2\pi}\sum_{j=0}^{K}g_j\bigg(\frac{B}{T_1},\frac{B}{T_2}\bigg)\tau_0^{j+1}\bigg)
+ O_{A}(  N^{-A}).
\ena
This proves the second statement of the lemma.

(3) For $B\gg T_1^{1+\varepsilon}$, the proof is similar as that of (2) and we will be brief.
In this case, the formula \eqref{The integral:3} still holds.
Thus repeated integration by parts shows that $\Psi_{\varphi}^{\pm}(x,n,q)$
is negligibly small unless $Nx\asymp T_1|T_2|$.
Note that the total variation of $V_0^{\pm}(\tau)$ is bounded by 1 and the second
derivative of the phase function is of size $B$.
By the second derivative test in Lemma \ref{lem: 2st derivative test, dim 1}, we have
\bna
\Psi_{\varphi}^{\pm}(x,n,q)\ll (Nx)^{1/2}.
\ena
This finishes the proof of the lemma.

\end{proof}

Next we prove Lemma \ref{integral:lemma}.
\begin{proof}[Proof of Lemma \ref{integral:lemma}]
The proof is similar to \cite[Lemma 4.2]{HSZ}.
Recall \eqref{H-integral} which we relabel as
\bea\label{H-relabel}
\mathcal{H}(x)=\int_{\mathbb{R}}
\omega\left(\xi\right)
\mathfrak{I}^{\pm}\left(C^2T_1|T_2|\xi/N,n_1,q\right)
\overline{\mathfrak{I}^{\pm}\left(C^2T_1|T_2|\xi/N,n_2,q\right)}
\, e\left(-x\xi\right)\mathrm{d}\xi,
\eea
where by \eqref{I-definition},
\bea\label{I-definition2}
\mathfrak{I}^{\pm}(C^2T_1|T_2|\xi/N,n,q)=V_{\natural}^{\pm}(\tau_*)
e\bigg(\frac{n^{1/2}N^{1/2}}{\pi C\xi^{1/2}}
+\frac{B}{2\pi}\sum_{j=1}^{K}g_j\bigg(\frac{B}{T_1},\frac{B}{T_2}\bigg)
\left(\frac{q}{2C}\right)^{j+1}\xi^{-(j+1)/2}\bigg).
\eea
Trivially, one has
\bna
\mathcal{H}(x)\ll 1.
\ena
This proves the first statement of Lemma \ref{integral:lemma}.

Plugging \eqref{I-definition2} into \eqref{H-relabel}, we obtain
\bna
&&\mathcal{H}(x)=
\int_{\mathbb{R}}
\omega\left(\xi\right)V_{\natural}^{\pm}(\tau_*)\overline{V_{\natural}^{\pm}(\tau_*')}
e\left(-x\xi+\frac{(n_1^{1/2}-n_2^{1/2})N^{1/2}}{\pi C\xi^{1/2}}\right)\\
&&\qquad\qquad\times e\left(\frac{1}{2\pi}\sum_{j=1}^{K}
\bigg(Bg_j\bigg(\frac{B}{T_1},\frac{B}{T_2}\bigg)
-B'g_j\bigg(\frac{B'}{T_1},\frac{B'}{T_2}\bigg)\bigg)
\left(\frac{q}{2C}\right)^{j+1}\xi^{-(j+1)/2}\right)
\mathrm{d}\xi,
\ena
where $\tau_*$ are as in \eqref{stationary point-1} with $B=2 n_1^{1/2}N^{1/2}/q$,
$\tau_*$ is defined in the same way but with $B'=2 n_2^{1/2}N^{1/2}/q$.
Note that the first derivative of
the phase function in the above integral equals
\bea\label{1st phase function}
&&-x-\frac{(n_1^{1/2}-n_2^{1/2})N^{1/2}}{2\pi C\xi^{3/2}}\nonumber\\&&
-\frac{1}{2\pi}\sum_{j=1}^{K}\bigg(\frac{j+1}{2}\bigg)
\bigg(Bg_j\bigg(\frac{B}{T_1},\frac{B}{T_2}\bigg)
-B'g_j\bigg(\frac{B'}{T_1},\frac{B'}{T_2}\bigg)\bigg)
\left(\frac{q}{2C}\right)^{j+1}\xi^{-(j+3)/2}
\eea
which is $\gg |x|\gg N^{\varepsilon}$
if $|x|\gg N^{1+\varepsilon}\Xi/(CQ)$
since $n_i\sim N\Xi^2/Q^2$, $i=1,2$ and \eqref{zeta range}.
Then repeated integration by parts shows that
the contribution from $x\gg  N^{1+\varepsilon}\Xi/(CQ)$ is negligible.
Thus the second statement of Lemma \ref{integral:lemma} is clear.

Moreover, the second term in \eqref{1st phase function}
is of size
\bna
\frac{N^{1/2}}{C}|n_1^{1/2}-n_2^{1/2}|\asymp \frac{Q}{C\Xi}|n_1-n_2|
\ena
since $n_i\sim N\Xi^2/Q^2$, $i=1,2$. Thus repeated integration by parts shows that
$\mathcal{H}(x)$ is negligibly small unless $|x|\asymp \frac{Q}{C\Xi}|n_1-n_2|$.
Now by applying the second derivative test
in Lemma \ref{lem: 2st derivative test, dim 1}, we infer that for $x\neq 0$,
\bna
\mathcal{H}(x)\ll |x|^{-1/2}.
\ena
This proves (3).

Finally, for $x=0$, using the fact $g_j(y_1,y_2)$ are some homogeneous polynomials of degree $j$
and satisfy
$g_j(y_1,y_2)\ll_j y_2^{j}$ for any integer $j\geq 1$ and the identity
$a^{j+1}-b^{j+1}=(a-b)(a^j+a^{j-1}b+\cdots+ab^{j-1}+b^j)$, one sees that, for $j\geq 1$,
\bna
&&Bg_j\bigg(\frac{B}{T_1},\frac{B}{T_2}\bigg)
-B'g_j\bigg(\frac{B'}{T_1},\frac{B'}{T_2}\bigg)\\
&=&B\sum_{j_1+j_2=j}\left(\frac{B}{T_1}\right)^{j_1}\left(\frac{B}{T_2}\right)^{j_2}
-B'\sum_{j_1+j_2=j}\left(\frac{B'}{T_1}\right)^{j_1}\left(\frac{B'}{T_2}\right)^{j_2}\\
&=&(B^{j+1}-B'^{j+1})\sum_{j_1+j_2=j}T_1^{-j_1}T_2^{-j_2}\\
&\ll&|B-B'|(B_j+B'^j)T_2^{-j}\\
&\ll& |B-B'|N^{-\varepsilon}.
\ena
Thus the first derivative
of the phase function in \eqref{1st phase function} is
\bna
\gg |B-B'|\asymp \frac{Q}{C\Xi}|n_1-n_2|.
\ena
By repeated integration by parts, $\mathcal{H}(0)$
is negligible small unless $|n_1-n_2|\ll C\Xi N^{\varepsilon}/Q$.
Since $\Xi\ll N^{\varepsilon}$ and $C\ll Q$, we have that
$\mathcal{H}(0)$
is negligibly small unless $|n_1-n_2|\ll N^{\varepsilon}$.
This completes the proof of Lemma \ref{integral:lemma}.
\end{proof}

\end{document}